\documentclass[fleqn,a4paper, twoside]{article}

\usepackage[
	top = 1.5 in, 
	bottom = 1.25 in,
	left = 1.25 in, 
	right = 1.25 in]{geometry}
\usepackage{scrextend}
\changefontsizes{11pt}

\usepackage{mathptmx}
\usepackage[T1]{fontenc}

\usepackage[all]{nowidow}

\usepackage[utf8x]{inputenc}
\usepackage{amsfonts,amssymb,amsmath}
\usepackage{amsrefs}
\usepackage{url}
\usepackage[new]{old-arrows} %% LONG HOOK ARROW!
\usepackage{array}

\BibSpec{article}{%
    +{}  {\PrintAuthors}                {author}
    +{,} { \textit}                     {title}
    +{.} { }                            {part}
    +{:} { \textit}                     {subtitle}
    +{,} { \PrintContributions}         {contribution}
    +{.} { \PrintPartials}              {partial}
    +{,} { }                            {journal}
    +{}  { \textbf}                     {volume}
    +{}  { \PrintDatePV}                {date}
    +{,} { \issuetext}                  {number}
    +{,} { \eprintpages}                {pages}
    +{,} { }                            {status}
    +{,} { \url}                        {url}    % <---- ADDED
    +{,} { \PrintDOI}                   {doi}
    +{,} { available at \eprint}        {eprint}
    +{}  { \parenthesize}               {language}
    +{}  { \PrintTranslation}           {translation}
    +{;} { \PrintReprint}               {reprint}
    +{.} { }                            {note}
    +{.} {}                             {transition}
    +{}  {\SentenceSpace \PrintReviews} {review}
}

\usepackage{graphicx}
\usepackage[
	british]{babel}
\usepackage{caption}
\usepackage{mathdots}
\usepackage{mathtools} 
\usepackage{amsmath}
\usepackage{amsfonts}
\usepackage{amssymb}

\usepackage{pgf,tikz}
\usepackage{tikz-cd}
\pgfdeclarelayer{bg}    % declare background layer
\pgfsetlayers{bg,main}  % set the order of the lay
\usetikzlibrary{calc}
\usetikzlibrary{matrix,arrows}
\usepackage{stackrel}
\usepackage[shortlabels]{enumitem} %Para listas estilo Weibel
\usepackage{stmaryrd} %Para double brackets
\usepackage{setspace} %Para espaciado de renglones
\spacing{1.15}
\usepackage{etoolbox}
\usetikzlibrary{trees}
\usepackage{enumitem}
\setlist[enumerate]{label= (\arabic*)}

\patchcmd{\section}{\normalfont}{\normalfont\large}{}{}

\renewenvironment{abstract}{%
\small\begin{center}
\begin{minipage}{.9\textwidth}
}
{\par\noindent\end{minipage}\end{center}\vspace{3 em}}
\makeatletter
\renewcommand\@maketitle{%
\hfill
\begin{center}\begin{minipage}{0.9 	\textwidth}
\centering
\vskip 2em
\let\footnote\thanks 
{\LARGE \@title \par }
\vspace{-.5 em}
\vskip 1 em
{\large \@author \par}
\vspace{3.5 em}

\end{minipage}\end{center}
\par
}
\makeatother
 \usepackage{fancyhdr}
\usepackage{wrapfig}

 \definecolor{newcol}{rgb}{0,0,0}
 \definecolor{niceblue}{rgb}{0.0, 0.18, 0.39}
\DeclareTextFontCommand{\new}{\color{newcol}\em}

\makeatletter
\AtBeginDocument{%
  \let\[\@undefined
  
\DeclareRobustCommand{\[}{\begin{equation}}%
  \let\]\@undefined
  
\DeclareRobustCommand{\]}{\end{equation}}%
}
\makeatother 
\mathtoolsset{showonlyrefs,showmanualtags}

\usepackage{amsthm}
\usepackage{thmtools}

\usepackage{stmaryrd}

\usepackage{amsthm}
\usepackage{thmtools}
\newtheoremstyle{mytheorem}
  {\topsep}   % ABOVESPACE
  {\topsep}   % BELOWSPACE
  {\itshape}  % BODYFONT
  {0pt}       % INDENT (empty value is the same as 0pt)
  {\bfseries\color{newcol}} % HEADFONT
  {}         % HEADPUNCT
  {5pt plus 1pt minus 1pt} % HEADSPACE
  {}          % CUSTOM-HEAD-SPEC
\theoremstyle{mytheorem}
\newtheorem{theorem}{Theorem}[section]
\newtheorem{cor}[theorem]{Corollary}
\newtheorem{lemma}[theorem]{Lemma}
\newtheorem{prop}[theorem]{Proposition}

\theoremstyle{mytheorem}
\newtheorem*{theorem*}{Theorem}
\newtheorem*{lemma*}{Lemma}
\newtheorem*{corollary*}{Corollary}
\newtheorem*{conjecture*}{Conjecture}

\DeclareTextFontCommand{\new}{\color{newcol}\bf\em} %%NEW DEFINITION

\usepackage{sectsty}
\chapterfont{\color{newcol}}  % sets colour of chapters
\sectionfont{\color{newcol}}  % sets colour of sections
\subsectionfont{\color{newcol}}  % sets colour of sections

\usepackage{sectsty}
\chapterfont{\color{newcol}}  % sets colour of chapters
\sectionfont{\color{newcol}}  % sets colour of sections
\subsectionfont{\color{newcol}}  % sets colour of sections

\theoremstyle{plain} % just in case the style had changed
\newcommand{\thistheoremname}{}
\newtheorem{genericthm}[theorem]{\thistheoremname}

\newcommand{\imor}{\interleave\kern-.45em\longrightarrow}

\newcommand{\DCSH}{\mathsf{Cog_\infty}}

\newcommand{\Mod}{\mathsf{Mod}}

\newcommand{\Der}{\operatorname{Der}}
\newcommand{\Coder}{\operatorname{Coder}}
\newcommand{\uk}{\underline{\kk}}

\newcommand{\vv}{\vert}
\newcommand{\HH}{\mathrm{HH}}

\newenvironment{tenumerate}{
\begin{enumerate}
  \setlength{\itemsep}{0pt}
  \setlength{\parskip}{0pt}
}{\end{enumerate}}

\definecolor{comcol}{rgb}{0.00, 0.00, 0.00}
\usepackage[pdftex, 
			colorlinks,
			bookmarks = true,
			bookmarksnumbered = true,
			]{hyperref}
			
\hypersetup{colorlinks,
 linkcolor={niceblue},
 citecolor={niceblue},
 urlcolor={niceblue}}

\newenvironment{titemize}{
\begin{itemize}
  \setlength{\itemsep}{0pt}
  \setlength{\parskip}{0pt}
}{\end{itemize}}

\usepackage{xargs}                      
\usepackage[colorinlistoftodos,prependcaption,textsize=small]{todonotes}
\newcommandx{\unsure}[2][1=]{\todo[linecolor=blue,backgroundcolor=blue!25!white,bordercolor=blue,#1]{#2}}
\newcommandx{\change}[2][1=]{\todo[linecolor=blue,backgroundcolor=blue!25,bordercolor=blue,#1]{#2}}
\newcommandx{\info}[2][1=]{\todo[linecolor=OliveGreen,backgroundcolor=OliveGreen!25,bordercolor=OliveGreen,#1]{#2}}
\newcommandx{\improvement}[2][1=]{\todo[linecolor=Plum,backgroundcolor=Plum!25,bordercolor=Plum,#1]{#2}}
\newcommandx{\thiswillnotshow}[2][1=]{\todo[disable,#1]{#2}}

\usepackage{multicol}
\newcommand{\DA}{\mathsf{Alg}}
\newcommand{\Gmod}{\mathsf{GMod}}

\newcommand{\NN}{\mathbb N}
\newcommand{\kk}{\Bbbk}

\newcommand{\Ext}{\operatorname{Ext}}
\newcommand{\Tor}{\operatorname{Tor}}

\newcommand{\Addresses}{{% additional braces for segregating \footnotesize
  \bigskip
  \footnotesize

  \noindent\textsc{School of Mathematics, Trinity College, Dublin 2, Ireland}\par\nopagebreak
  \noindent\textit{E-mail address:} \texttt{pedro@maths.tcd.ie}
  }}
  
\usepackage{titletoc}

\titlecontents{chapter}
[0.2em] %
{\bigskip}
{\makebox[2em][r]{\thecontentslabel.}\hspace{0.333em}}%\thecontentslabel
{\hspace*{-2em}}
{\hfill\contentspage}[\smallskip]

\titlecontents{section}% <section>
[0.5em] % <left>
{\small}% <above-code>
{\thecontentslabel.\hspace{3pt}}%<numbered-entry-format>; you could also 
{}% <numberless-entry-format>
{\enspace\titlerule*[0.5pc]{.}\contentspage}%<filler-page-format>
\titlecontents*{subsection}% <section>
[1 em]% <left>
{\footnotesize}% <above-code>
{\thecontentslabel. \hspace{3pt}}% <numbered-entry-format>; you could also 
{}% <numberless-entry-format>
{}% <filler-page-format>
[ --- \ ]% <separator>
[]% <end>
\makeindex

\title{\textbf{Minimal models for monomial algebras}}
\author{\textsc{Pedro Tamaroff}}
\date{}

\begin{document}
\maketitle

\begin{abstract}
We give, for any monomial algebra $A$, an explicit description of its minimal 
model, which also provides us with formulas for a canonical $A_\infty$-structure
on the Ext-algebra of the trivial $A$-module. We do this by exploiting the 
combinatorics of chains going back to works of Anick, Green, Happel and Zacharia,
and the algebraic discrete Morse theory of J\"ollenbeck, Welker and Sk\"oldberg. We then
show how this result can be used to obtain models
for algebras with a chosen Gröbner basis, and briefly
outline
how to compute some classical homological invariants with it. 

\medskip
{\textbf{MSC 2020:} 16E05, 16E40, 16E45, 18G15, 18N40}.
\end{abstract}
\thispagestyle{empty}

%\tableofcontents

%\pagebreak

\section{Introduction}\label{sec:intro}
Understanding $A_\infty$-structures associated to differential graded 
associative (dga) algebras is central to 
understanding in turn, the homotopy category of the category $\DA$ of 
dga algebras. More precisely, one can, in principle, compute in the homotopy 
category of $\DA$ by considering the category of quasi-free dga algebras or, 
equivalently, $A_\infty$-coalgebras, modulo the usual relation of homotopy 
between morphisms in $\DA$: the quasi-free dga algebras are
cofibrant in $\DA$, where the weak equivalences are the quasi-isomorphisms 
and the fibrations are the degree-wise epimorphisms; see~\cite{Hinich} and \cite{HTHA}*{Proposition 1.5}.

In particular, we may use $A_\infty$-coalgebras to understand usual
(non-dg) associative algebras. For any augmented algebra $A$ over a field 
$\kk$ one can produce, from the bar construction $BA$ of $A$, the class of minimal $A_\infty$-coalgebra 
structures on $\Tor_A(\kk,\kk)$. Among other things, these 
determine $A$ up to isomorphism, and may be used to compute its Hochschild 
cohomology or obtain the minimal model of $A$; see \cite{Kell,Has}. 
The explicit computation of such higher structures is therefore of 
interest. The machinery of Gr\"obner bases and homological perturbation theory 
suggest that a possible step towards solving this problem is to first 
obtain an answer for monomial algebras. In this paper we provide a complete 
description of a canonical minimal $A_\infty$-coalgebra structure on 
$\Tor_A(\kk,\kk)$ for a monomial algebra $A$ in terms of the combinatorics of its chains. 
Equivalently, we completely describe a minimal model of $A$ as the $\infty$-cobar
construction $\Omega_\infty\!\Tor_A(\kk,\kk)$. The results extend without 
modification to describe minimal models of monomial quiver algebras in terms of the combinatorics of their chains; see \cite{GHZ}.

Concretely, let $\gamma$ be a basis element of $\Tor_A^{r+1}(\kk,\kk)$, 
represented by an Anick chain of length $r\in\NN$ and let us take $n\in
\NN_{\geqslant 2}$. A \new{decomposition} of $\gamma$ is a tuple $
(\gamma_1,\ldots,\gamma_n)$ of chains of respective lengths $(r_1,\ldots,r_n)$ satisfying $r_1+\cdots
+r_n=r-1$ and whose concatenation, in this order, is~$\gamma$. Our result
is the following.

\begin{theorem*} For each monomial algebra $A$ there is a minimal model
$B\longrightarrow A$ where $B$ is the $\infty$-cobar construction on $\Tor_A(\kk,\kk)$. The differential $d$ is such that for a chain $\gamma\in\Tor_A(\kk,\kk)$,
\begin{align*} d \gamma  = -\sum_{n\geqslant 2}
	(-1)^{\binom{n+1}{2}+|\gamma_1|} 
	 	 \gamma_1 \cdots \gamma_n, 
	 		\end{align*}
where the sum ranges through all possible decompositions of $\gamma$.\qed
	 \end{theorem*}
\noindent
This recovers, in particular, the results in \cite{GreenZ} describing cup
products in $\Ext_A(\kk,\kk)$ for a monomial quiver algebra $A$ using a multiplicative 
basis of chains, and the results in~\cite{HeLu} describing the 
$A_\infty$-algebra structure of $\Ext_A$ for monomial algebras which 
are $p$-Koszul. 

\bigskip

The paper is organised as follows. In 
Section~\ref{sec1:rec} we recall the relevant definitions and 
constructions from homological and homotopical algebra to be used
throughout the paper. In particular, we recall the essentials from
\cite{Anick}, the central results of algebraic discrete Morse theory
presented in \cite{ADMT}, and the dual version of the homotopy transfer theorem 
for $A_\infty$-algebras from \cite{Markl}. 
In Section~\ref{sec2:homcomb} we use the results of \cite{ADMT} to 
produce a homotopy retract datum from the bar construction of $A$ to $\Tor_A(\kk,\kk)$
and therefore a minimal $A_\infty$-coalgebra structure on $\Tor_A(\kk,\kk)$, 
which we describe explicitly in Section~\ref{sec3:mainres} in terms of 
decompositions of Anick chains into concatenations of smaller chains, 
and we note that our results generalize directly to the case of quiver algebras 
defined by monomial relations. Finally, in Section~\ref{sec4:Apps}
we  outline how to exploit the results of Section~\ref{sec3:mainres} to 
compute invariants of algebras and models of algebras with a 
chosen Gr\"obner basis, and briefly explain how our model has been successfully used
to study the support variety theory of monomial algebras~\cite{Gorenstein}. 

\bigskip

We fix once and for all a field $\kk$. 
All unadorned $\hom$ and $\otimes$, which denote the usual bifunctors on 
graded vector spaces, will be taken with respect to $\kk$. We let $\kk s^{-1}$ be 
the graded vector space concentrated in degree $-1$, where it is one 
dimensional, and write $s^{-1}$ for its generator. If $V$ is a graded
vector space, we write $s^{-1}V$ for $\kk s^{-1}\otimes V$, and denote 
$s^{-1}\otimes v$ by $s^{-1}v$, and write $V^\vee$ for the graded dual of~$V$.
\bigskip

\subsection*{Acknowledgements} I am pleased to thank Vladimir
Dotsenko for the constant support during the preparation of these
notes, and in particular for useful remarks and corrections.
Shortly before releasing the first draft of this paper, we became 
aware that a formula for the $A_\infty$-structure on the Yoneda 
algebra of any monomial quiver algebra was found by Chuang and 
King around 2005; they managed to guess a formula for that 
structure and check the Stasheff identities, and  then used a 
twisting cochain argument to establish that their structure is in 
the correct isomorphism class. We thank Joe Chuang for providing 
a copy of their unpublished manuscript. Finally,
we thank an anonymous referee for very useful comments
and suggestions that significantly improved the
exposition of the paper. 
 
\section{Recollections}\label{sec1:rec}
\textcolor{comcol}{
We recall that a monomial algebra over $\kk$
is a quotient of a free algebra $TV$ on a
finite dimensional $\kk$-module $V$ by an 
ideal generated by finitely many monomials.
We will follow the conventions and  
definitions from~\cite{LV}, and we refer the
reader to it for the essentials on weight graded differential graded algebras. In particular,
we follow their convention regarding gradings: 
if a dga algebra has an additional
grading by weight, efer to it as the \emph{weight grading}, as opposed to the
\emph{homological grading}, to avoid any confusion. 
Hence, the grading of a
graded algebra with a zero differential is
homological: a weight grading is, for us, an
\emph{extra grading} coming from some additional
structure so that, for example, our monomial
algebras will be weight graded but concentrated 
in homological degree zero. 
}

\smallskip
As explained in the introduction, we will completely describe, for
a given monomial algebra $A$, a minimal model $B\longrightarrow A$.
Recall this is a quasi-isomorphism onto $A$ from a quasi-free dga algebra
$B$ (that is, a dga algebra whose underlying weight graded algebra is free) whose differential satisfies the Sullivan condition \cite{LV}*{\S B.6.8}; although this condition
is necessary to have a well-behaved model of 
$A$, it will not be central to our exposition,
and the fact that our model satisfies this 
condition will be immediate to check.

Although this gives us, a priori, information about $A$ in the homotopy
category of $\DA$, there is a rich feedback loop between homotopical and
homological algebra, already present in the original work of Quillen,
and successfully pursued in \cite{Hinich,Kell}, among others. Without
going into details, we will content ourselves with giving a few examples:
\begin{titemize}
\item A model of $A$, that is, its homotopy type, 
can be computed
entirely by homological and perturbative methods, 
starting with homological invariants of it.
\item From this one may compute the Hochschild homology and cohomology of $A$
and, in particular, obtain information about the derived category
of its representations, and the representations of its enveloping algebra.
\item In fact,
the homotopy type of the dg Lie algebra of 
derivations of a model determines the deformation theory of~$A$. 
\end{titemize}

\noindent
All results of this paper can be proved for quiver algebras with
monomial relations; for readability, we present all arguments in the
case of associative algebras (that is, for one-vertex quivers) and
then merely state the corresponding generalization. 

\subsection{Bar resolution and Tor}

Let $A$ be a weight graded $\kk$-algebra. 
Observe that if $\Omega_\infty C\longrightarrow A$ is a
minimal model of $A$, then the space of indecomposable elements $C$ of 
$\Omega_\infty C$ can be identified with $\Tor_A(\kk,\kk)$ and is, in fact, the Quillen 
homology of $\Omega_\infty C$: it will become apparent in what follows that
our choice of basis for $\Tor_A(\kk,\kk)$, that of Anick chains, will be central in
describing our choice of minimal model of $A$.

{{}} Write ${}_A\Mod$ 
and $\Mod_A$ for the respective categories of left and right $A$-modules. 
The bifunctor $-\otimes_A - : \Mod_A\times {\vphantom{m}}_A\Mod \longrightarrow
{}_\kk\Mod$ gives us, upon derivation, the classical bifunctor $\Tor_A(-,-):
\Mod_A\times {}_A\Mod\longrightarrow {}_\kk\Gmod$ 
\textcolor{comcol}{to the category of \emph{graded} $\kk$-modules}, defined as follows. 
For $M\in\Mod_A$ and $N\in {}_A\Mod$, let us pick respective projective 
resolutions $P\to M$ and $Q\to N$ in  $\Mod_A$ and ${}_A\Mod$. The  
diagram
$$
	P\otimes_A N \longleftarrow  
		P\otimes_A Q \longrightarrow 
			M\otimes_A Q
				$$
connects the above three complexes by natural quasi-isomorphisms, 
up to our choice of resolutions, 
and their homology is the graded $\kk$-module $\Tor_A(M,N)$. Let us remark
that $\Tor_A(M,N)$ is usually denoted by $\Tor^A(M,N)$ but that for typographical 
purposes we will instead write it $\Tor_A(M,N)$. When $A$ is connected or, more
generally, augmented, we will write $\Tor_A$ for $\Tor_A(\kk,\kk)$, where $\kk$
is made into an $A$-module via the augmentation $A\longrightarrow \kk$.

{{}} There is a particularly useful way we can construct such
bifunctor following the definition above. Concretely,
if $R\longrightarrow A$ is any projective resolution of the $A$-bimodule $A$,
then the homology of the complex $M\otimes_A R \otimes_A N$ is
$\Tor_A(M,N)$. The advantage of this is we need only choose one
resolution, namely that of $A$ as an $A$-bimodule, 
to obtain resolutions for every
left or right $A$-module, and we now fix this choice. 
Define $B(A,A,A)$, the \new{double-sided bar resolution of $A$},
to be the chain complex such that for each $n\in\mathbb N_0$, we have $B_n(A,A,A) = A\otimes \overline{A}^{\otimes n}\otimes A$,
the free $A$-bimodule with basis $\overline{A}^{\otimes n}$, \textcolor{comcol}{where $\overline{A}$ is the kernel of the augmentation $A\longrightarrow \kk$.} 

For each such integer, denote a generic bimodule basis element in degree $n$ by 
$[a_1\vv\cdots\vv a_n]$.
Its differential is then given by
\[  -a_1[a_2\vv \cdots\vv a_n] 
	+\sum_{i=1}^{n-1}  (-1)^{i-1}[{a_1}\vv\cdots\vv {a_i}a_{i+1}\vv\cdots
	\vv a_n]
	+ (-1)^{n-1} [{a_1}\vv\cdots\vv {a_{n-1}}]a_n
		\] 
and is extended $A$-bilinearly. In particular, if $n=0$ we have
$B_0(A,A,A) =A\otimes A$ and there is an augmentation
$B_0(A,A,A)\to A$ given by multiplication which renders the
augmented complex $B(A,A,A)\to A$ contractible both as a complex
of left and as a complex of right $A$-modules. From this it 
follows that if $M$ is right $A$-module and $N$ a left $A$-module,
the complex
$B(M,A,N) := M\otimes_A B(A,A,A)\otimes_A N$  computes
$\Tor_A(M,N)$.
 
{{}} From now on we assume that $A$ is connected, which makes
it naturally augmented, and endows $\kk$ with a trivial $A$-module
structure on both sides, 
in which elements of positive degree act by zero. From
the previous remarks it follows that the complex $B(\kk,A,A)$ 
is a resolution of the right $A$-module $\kk$ by free right 
$A$-modules, which we will denote by $B(A,A)$ and call the \new{bar resolution of $\;\kk$}, so that $\Tor_A(\kk,\kk)$ may be 
computed as the homology of the complex $B(\kk,A,\kk)$, which we simply denote
by $BA$ and call the \new{bar construction of $A$}. Concretely, 
we have for each $n\in\mathbb N_0$ a natural isomorphism 
$(BA)_n \to \overline{A}^{\otimes n}$, which we consider 
an identification, with differential given on basis elements 
$[a_1\vv\cdots\vv a_n]$ by
\[
  d[a_1\vv\cdots\vv a_n] = \sum_{i=1}^{n-1} (-1)^{i-1}
  	[{a_1}\vv\cdots
  		\vv {a_i}a_{i+1}
  			\vv \cdots \vv a_n]. \]

{{}} The complex $BA$ admits a diagonal $\Delta_2' : BA
\longrightarrow BA\otimes BA$ given by deconcatenation, 
that makes it into a non-unital dga coalgebra. Concretely, on 
basis elements  $[a_1\vv\cdots\vv a_n]$ of degree $n\in\mathbb N_0$
we have that 
\[\Delta_2'[a_1\vv\cdots\vv a_n] = 
	\sum_{i=1}^{n-1} [a_1\vv\cdots\vv a_i]\otimes
						[a_{i+1}\vv\cdots\vv a_n].\]
						
\subsection{Anick's resolution}
% and Govorov's formulas}

{{}} In his celebrated article \cite{Anick}, Anick constructs an $A$-free
resolution of the trivial module for any augmented algebra $A$ 
equipped with a Gr\"obner basis. This construction is generalized for
quiver algebras in \cite{AnickG}, where the authors
use the notion of chains for such algebras from \cite{GHZ}.
Since we will use the description of $\Tor_A$ by means
of Anick's resolution,
let us quickly recall his results.

Let us write $S$ for a set of generators of $A$, the
variables, and let $f : \kk\langle S\rangle \longrightarrow A$
be the quotient map by the ideal of relations of $A$, which is
a map of augmented $\kk$-algebras. We weight grade $A$ by the length of
a monomial, and give $S$ a total order. This induces on the monoid 
of monomials $M_S$ on $S$ a well ordering in such a way that $m < m'$ if $|m|<|m'|$ \textcolor{comcol}{(here $|m|$ is the length of a monomial)}, or if $m$ and $m'$ are of the same length but 
$m<m'$ in the lexicographical order, 
\textcolor{comcol}{also known as the dictionary order, induced from the total order of the letters in $S$: if $m =x_1\cdots x_t$
and $m' =y_1\cdots y_t$ are monomials of the same length, then $m<m'$ if $x_i<y_i$ for the first 
index $j\in [t]$ such that $x_j\neq y_j$.} Given monomials $u,v\in M_S$, say \new{$v$ is a divisor of $u$} if $u = u'vu''$ for 
monomials $u',u''$, to obtain a partial ordering $\subseteq$ on $M_S$. 
A subset $I$ of $M_S$ is an \new{order ideal of monomials} if it is 
a lower set for $\subseteq$. \textcolor{comcol}
{Note that the
reverse ordering associated to $\subseteq$
is sometimes used in the literature: for us,
$v\subseteq u$ means that $v$ is a divisor of
$u$.}
It is readily checked by induction that
the set $N = \{x\in M_S : f(x)\notin \langle f(y) : y<x\rangle\}$ is
an order ideal of monomials, and that $f(N)$ is a basis of $A$
as a $\kk$-module. 

{{}} From $N$ Anick extracts the basic building
blocks for his resolution, the \new{obstructions}. Concretely, let
$V$ consist of those $x\in M_S$ that are not in $N$, but
all $y\subsetneq x$ are in $N$. These are simply the maximal 
elements of the order ideal $N$, and thus form an anti-chain.
Its elements are the obstructions. From the definitions it 
follows that an element is in $M_S\smallsetminus N$ precisely when it 
contains as a divisor an obstruction. In case $A$ is monomial,
$N$ consists of those monomials that contain no monomial relation
as a divisor, and the obstructions are the minimal relations of $A$. 
Now set $V^{-1} = \kk$, $V^0 = \kk S$ and $V^1=\kk V$, to 
begin to construct a right $A$-free resolution
\[ \cdots \longrightarrow V^2\otimes A
	\stackrel{\delta_2}
	\longrightarrow V^1\otimes A
	\stackrel{\delta_1}\longrightarrow V^0\otimes A
	\stackrel{\delta_0}\longrightarrow V^{-1}\otimes A \longrightarrow 0 \]
of~$\kk$. For each $n\in\mathbb N$ we now obtain a vector space $V^n$ 
with a basis of monomials, called the $n$-chains, in the following way.
An \new{$n$-prechain} is a monomial $x_{i_1}\cdots x_{i_t}$ in $B$ for
which there exist strictly increasing sequences of integers 
$(a_1,\ldots,a_n)$ and $(b_1,\ldots,b_n)$ with $a_1=1$ and $b_n=t$ 
such that the sequences are interlaced, meaning that $a_{i+1}\leqslant 
b_i$ for each $i\in \{1,\ldots, n-1\}$, and such that for each $j
\in\{1,\ldots,n\}$, the monomial $x_{i_{a_j}}\cdots x_{i_{b_j}}$ is 
an obstruction.

{{}} In particular, the collection of $1$-prechains, which coincides with that of $1$-chains, is a basis for $V^1$. 
We say an $n$-prechain is an \new{$n$-chain} if the two previous sequences 
may be chosen so that $x_{i_1}\cdots x_{i_s}$ is not an $m$-prechain
for any $s<b_m$ and $m\in \{1,\ldots,n\}$. Plainly, a chain is
a prechain that satisfies a minimality condition regarding the 
overlappings between the obstructions that constitute it.
It is readily verified
that in this case these two sequences are uniquely determined,
there is a unique $s=b_{n-1}<t$ such that $x_{i_1} \cdots x_{i_s}$ is 
an $(n-1)$-chain and the tail $x_{i_s+1}\cdots x_{i_t}$ contains
no divisor that is an obstruction. This is the key observation
to construct a sequence of boundary maps $(\delta_n : V^n\otimes A \longrightarrow V^{n-1}\otimes A)_{n\geqslant 2}$ such that
\[\delta_n(x_{i_1} \cdots x_{i_t}) = 
x_{i_1} \cdots x_{i_s}\otimes x_{i_s+1}\cdots x_{i_t} + 
	\text{lower terms.}\]
If $A$ is monomial, there are no lower terms in the differential 
and this resolution is minimal, so that for each $n\in\mathbb N$, 
$\Tor_A^{n+1}$ is identified with the vector space $V^n$ with basis consisting of the
$n$-chains: this is the content of Lemma 3.3 in \cite{Anick}. 
Let us make the important remark that, in what follows, we adhere to 
such identification strictly: our main result depends critically on
using Anick chains to model $\Tor_A$.

Since it will be useful to illustrate
some of the rather technical constructions
that will follow, let us consider the 
algebra $J = T(x,y)/(x^2,y^2x-xy^2-xyx)$ and the
monomial 
algebra  $K = T(x,y,z)/(xy^2,y^2z)$. 
In the case of the algebra $J$ with
the order $y>x$, one can check that the specified
generator for its ideal of relations constitute
a Gröbner basis. In the case of $K$ this is
immediate for any order, 
since the relations are monomial
and there are no redundant relations. For $J$,
we obtain that for each $n\in\NN$, the set 
of $n$-chains is $\{x^{n+1},y^2x^n\}$, 
corresponding to $n$ overlappings of the relation
$x^2$ with itself, and of an overlapping of
$y^2x$ with $n-1$ copies of the relation $x^n$. 
In the case of $K$, we get finitely many Anick
chains: the $0$-chains are $\{x,y,z\}$, the
$1$-chains are the relations $\{xy^2,y^2z\}$,
and the $2$-chains are the overlappings 
$\{xy^2z,xy^3z\}$, and there are no other chains.

\subsection{Algebraic discrete Morse theory}\label{sec:ADMT}
 
{{}} Let $C$ be a non-negatively graded complex of $\kk$-modules. 
Fix a basis $X=\{X_t\}_{t\geqslant 0}$ of 
homogeneous elements of $C$, so that for each 
$t\in\NN_0$, the set $X_t$ is a basis of $C_t$. Given 
$c\in X$ we introduce the notation
\[ dc = \sum_{c'\in X} [c:c']c' \]
where $[c:c']\in \kk$. Let $G= G(C,X)$ be the directed weighted graph 
with vertices the set $X$ and with an edge $c\to c'$ if $c'$ 
appears in $dc$ with non-zero coefficient $[c:c']$ which 
is, in that case, the weight of $c\to c'$. A finite subset $M$ of 
edges of $G$ is a \new{Morse matching} if it satisfies the 
following \new{Morse conditions}:
\begin{enumerate}[label = \color{newcol}{\textbf{M\arabic*.}}]
\setlength{\itemsep}{0pt}
  \setlength{\parskip}{0pt}	
	\item Each vertex of $G$ is in at most one edge of $M$.
	\item The weights of edges of $M$ are invertible.
	\item The graph $G_M$ obtained by inverting the edges of $M$ in $G$  
	has no directed cycles. 
\end{enumerate}
If $c'\to c$ is a edge in $G_M$ with
$c\to c'\in M$, we set its weight to be $-[c:c']^{-1}$. In our situation the coefficients $[c:c']$ will be either $1$ or $-1$, which 
means \textbf{M2} is always satisfied. We write $X^M$ for 
the collection of vertices not appearing in $M$, which we call \new{critical}. 
Write $P(c,c')$ for the set of paths in $G_M$ from $c$ to $c'$, and assign a 
path the product of the weights of the edges it contains. Finally, write
$\Gamma(c,c')$ for the sum of all the weights of paths from $c$ to $c'$ in 
$G_M$. 

{{}} We define the \new{Morse complex of $C$ with respect to $M$}, 
which we denote by $C^M$, as the complex with  basis
the critical vertices $X^M$ and with differential
given, on basis elements, by
\[ dc = 
	\sum_{c'\in X^M_{t-1}}\Gamma(c,c') c' \]
whenever $c\in X_t$. The result of main interest to us in \cite{ADMT} 
is the following theorem, which shows how to produce a
homotopy retract datum from $C$ to $C^M$ given a Morse matching 
$M$ on $C$ relative to a basis $X$. 

\begin{theorem}\label{thm:mainadmt}
The complex $C^M$ is homotopy 
equivalent to $C$. More precisely, there
are maps $f : C \longrightarrow C^M$ and 
$g : C^M\longrightarrow C$
given on basis elements by
\[ f(c) = \sum_{c'\in X^M_t} 
\Gamma(c,c')c', \quad
 g(c) = \sum_{c'\in X_t} \Gamma(c,c')c' 
\]
for $c\in X_t$, respectively $c\in 
X_t^M$, which are inverse homotopy 
equivalences. In 
fact, $fg=1$ and $gf-1 = dh+hd$ 
where for a basis element 
$c\in X_t$,
\[ h(c) = \sum_{c'\in X_{t+1}} 
\Gamma(c,c')c'.\]
 \end{theorem}

{{}} Note that since for any two basis elements we have defined the
coefficient $\Gamma(c,c')$ as a sum through paths, it is
important that $M$ is finite for the theorem above to hold. We 
can, however, consider matchings $M$ of the complex $C$ if
$C$ is the colimit of a finite sequence of finite subcomplexes
$\{F^pC\}$ that is compatible with the matching, in the 
sense that $(F^pC)^M$ is a filtration by subcomplexes of
$C^M$. This last condition means $\Gamma(c,c')$ is well defined
and the last theorem extends in this situation. In particular,
we will consider the situation of $\NN$-multigraded complexes
such that each homogeneous subcomplex is finite, and in this case
the filtration by weight of tuples fulfills the condition above.

{{}} Let us note that in the homotopy $h$ we can only have a path
from an element in degree $t$ to one in degree $t+1$ if it is
given by a sequence of edges $e_0'e_1e_1'\cdots e_je_j'$
where $e_i'$ is an inverted edge of the matching and $e_i$
is a direct edge. Indeed, the first Morse condition forbids
a concatenation of inverted edges, which means we also
cannot have two consecutive non-inverted edges.
Finally, let us observe that if $c\in C^M$ is a cycle
then $g(c) = c$, that the last observation means that $h^2=0$,
and that $hg = 0$ and $fh=0$. Thus $(f,g,h)$ is a homotopy datum
that satisfies the side conditions, as defined in \ref{diag:HRD}.

\subsection{Anick's resolution via Morse theory}\label{sec:AADMT}
  
{{}} Let $A$ be a weight graded $\kk$-algebra presented by
generators $\{x_1,\ldots,x_n\}$ and ideal of 
relations~$I$, and assume that $\{f_1,\ldots,f_m\}$ 
is a reduced Gröbner basis with respect to a
fixed monomial order~$<$. Following~\cite{ADMT}, we show how
to obtain the Anick resolution of $A$ as the Morse 
complex of an acyclic matching on the normalized
bar resolution $B(A,A)$ of $\kk$, which we now denote 
more simply by $B$.  

{{}} Let $\text{in}(I)$ denote the ideal of leading terms of
elements of $I$, which is generated as an ideal
by the leading terms of the elements in 
$\{f_1,\ldots,f_m\}$. A monomial is \new{normal}
if it is not divisible by a leading term of an 
element in $\{f_1,\ldots,f_m\}$, and we write
$\mathsf{SM}$ for the collection of such monomials.
A monomial is \new{reducible} if it is not normal,
and we say that $uv=0$ \emph{minimally} if for every prefix 
$v'$ of $v$, the monomial $uv'$ is normal. 
The set $\mathsf{SM}$ is a basis of $A$ as a 
$\kk$-module. In particular, given two normal 
monomials $u$ and $v$ we can write
$uv = \sum_{w\in \mathsf{SM}} \lambda_w w$
where $|w|\leqslant |uv|$ for any $w$ 
with $\lambda_w\neq 0$. Observe that if $A$ 
is a monomial algebra, that is, if the 
relations of $A$ are given by monomials, the normal 
monomials are those that do not contain as a 
subword any monomial relation, and reducible 
monomials are zero in $A$.

{{}} We now define a Morse matching on $B$ by
induction. Recall that we denote a 
generic basis element of the bar resolution by 
$[a_1\vv \cdots \vv a_n]$.  Define $M_1$ to 
be the collection of edges of the form
$[x_i \vv w_1\vv w_2\vv \cdots 
\vv w_t ] 
 \longrightarrow
 	[ x_iw_1 \vv w_2\vv \cdots \vv 
w_t ]$.
The critical vertices $B^{(1)}$ with respect to 
$M_1$ are the variables $[x_1],\ldots,[x_n]$ in 
degree $1$, and those words $[x_1\vv w_1 \vv 
\cdots\vv w_t]$ of normal monomials such that 
$x_1w_1$ can be reduced. We proceed inductively to 
define $M_j$ for $j>1$. Having defined $M_{j-1}$, 
let $B^{(j-1)}$ be the set of critical vertices 
with respect to $M_1\cup \cdots \cup M_{j-1}$, and
define $E_j$ to be the set of edges that 
connect vertices of $B^{(j-1)}$. Suppose that 
\[
[x_{i_1} 
	\vv w_2
	\vv\cdots
	\vv w_{j-1}
	\vv u_1
	\vv u_2
	\vv w_{j+1}
	\vv \cdots
	\vv w_t ] 
\to
[x_{i_1} 
	\vv w_2
	\vv \cdots
	\vv w_{j-1}
	\vv w_j
	\vv w_{j+1}
	\vv \cdots
	\vv w_t]
\]
is an edge in $E_j$. In particular, $w_j 
= u_1u_2$. We say $e$ satisfies the \new{matching 
condition} if 
\begin{enumerate}[label = \color{newcol}{\textbf{B\arabic*.}}]
\setlength{\itemsep}{0pt}
  \setlength{\parskip}{0pt}
\item the monomial $u_1$ is a prefix of $w_j$,
\item the source of $e$ is in 
$B^{(j-1)}$ and,
\item for each prefix $v_1$ of $u_1$ and 
each $v_2$
such that $v_1v_2 = w_j$, the vertex 
\[
	[ 
	x_{i_1} \vv 
	w_2     \vv
	\cdots  \vv 
	w_{j-1} \vv
	v_1     \vv
	v_2     \vv 
	w_{j+1} \vv 
	\cdots  \vv 
	w_t ]
	 				\] 
is not in $B^{(j-1)}$. 
\end{enumerate}
We let $M_j$ be the collection of edges in $E_j$
that multiply monomials at the $j$th bar and
satisfy the matching condition. Then the set
$B^{(j)}$ is given by the variables $[x_i]$ in
degree $1$, the elements $[x_{i_1}\vv w ]$
such that $x_{i_1}w$ is a minimal monomial 
generating the ideal of leading terms of 
$I$, and the elements of the form
 $[ x_i \vv w_2\vv w_3\vv \cdots\vv w_t]$
such that for each prefix $u$ of $w_j$ 
the vertex $ 
  	[
  		x_i 	\vv
  		w_2		\vv
  		\cdots	\vv 
  		w_{j-1} \vv
  	    u 		\vv 
  	    w_{j+1}	\vv 
  	    \cdots 	\vv 
  	    w_t]
  	    	 $
is not in $B^{(j-1)}$ and the term $w_jw_{j+1}$ is
reducible. We set $M = \bigcup_{j\geqslant 1} M_j$
to be the desired Morse matching,
and let $B^M$ be the collection of critical vertices
with respect to $M$. 

\begin{lemma}\label{lemma:edges} 
\textcolor{comcol}{Assume that $A$ is a weight graded monomial $\kk$-algebra.} Let $j\in\NN$. 
The elements of $M_j$ consist of those edges
$ [x_i \vv u_1\vv\cdots\vv u_{j-1}\vv u_j\vv \cdots]\to
[x_i \vv u_1\vv\cdots\vv u_{j-1} u_j\vv\cdots ] $
such that $x_iu_1=u_1u_2 = \cdots = u_{j-2}u_{j-1}=0$ 
minimally and $ u_{j-1}u_j\neq 0$. Moreover, the collection 
$M$ is a Morse matching.
\end{lemma}
\begin{proof}
This is a particular case of \cite{ADMT}*{Lemma 4.2}.
\end{proof}

{{}} We now describe the critical vertices $B^M$. Let 
$m_{1},\ldots,m_{l}$ be minimal monomial
generators of the ideal of leading monomials of
$I$, such that for each $j\in\{1,\ldots,l\}$ we 
have $m_j = u_j v_j u_{j+1}$ where $u_1$ is a 
variable. We call the term
$ [
	u_1 	\vv 
	v_1u_2 	\vv 
	v_2u_3	\vv 
	\cdots 	\vv 
	v_{l}u_{l+1} 
	]$
\new{fully attached} if for all $j\in\{1,\ldots,l-1\}$ and each prefix $u$ of 
$v_{j+1} u_{j+2}$ the monomial $v_j u_{j+1}u$ is normal. We denote by $B_j$ 
the set of fully attached terms of degree $j\geqslant 2$ and let $B_1$ consist 
of the variables. We refer the reader to \cite{ADMT} for the
proof of the following lemma, valid for any
weight graded $\kk$-algebra with a Gröbner basis,
as in the beginning
of this section.

\begin{lemma} The fully attached tuples are exactly 
the critical vertices, and the complex $C^M$ is
the Anick resolution of $A$. In case $A$ is monomial,
the critical vertices are the variables $[x_1],\ldots,
[x_n]$ along with those terms 
$ [ x_i\vv u_1\vv\cdots \vv u_r ] $
where if we set $x_i = u_0$, we have that $u_ju_{j+1}=0$ minimally
for $j\in \{ 0,\ldots,r-1\}$.\qed
\end{lemma}

{{}} \textcolor{comcol}{
Let us briefly explain how one can go back
and forth from an Anick chain $\gamma$ to a
cycle in $\Tor_A$. If $\gamma$ is a $1$-chain 
then it is a monomial relation $x_1\cdots x_t$,
and the corresponding bar cycle (and critical
vertex) is $[x_1\vv x_2\cdots x_t]$. Now
suppose that $\gamma$ is an $n$-chain. Then
there exist a unique $(n-1)$-chain $\gamma'$ 
and a unique monomial $t$ such that $\gamma= \gamma' t$, and the bar term corresponding to
$\gamma$ is $[x_1\vv u_1\vv \cdots \vv u_l \vv
t ]$ where everything before $t$ is the bar
structure corresponding to $\gamma'$. In this
way, for example, if $\gamma$ is a $2$-chain
corresponding to an overlap of a relation $x_1\cdots x_t$ with another relation $x_j\cdots
x_t x_{t+1} \cdots x_s$, then the bar structure
of the corresponding critical vertex is
$[x_1\vv x_2\cdots x_t\vv x_{t+1}\cdots x_s]$. 
To illustrate, in the case of $J$,
for each $n\in \mathbb N$, the bar element
corresponding to $x^{n+1}$ is $[x\vv\cdots \vv x]$, where there are exactly $n$ bars, and the
bar element corresponding to $y^2x^n$ is
$[y\vv yx\vv x\vv \cdots \vv x]$ where, again,
we have $n$ bars. In the case of the algebra
$K$, the
 Anick chains correspond to the bar
elements }
\[[x],\,[y],\,[z],\,[x\vv y^2],\, [y\vv yz],\,[x\vv y^2\vv z] \text{ and } [x\vv y^2\vv yz].\] 
\subsection{Homotopy transfer theorem and \texorpdfstring{$A_\infty$}{Ainfty}-coalgebras}

{{}}\label{def:Stasheff} Recall that an $A_\infty$-coalgebra is a graded $\kk$-module 
$V$ along with sequence of locally finite maps $(\Delta_n : V
\longrightarrow V^{\otimes n})_{n\in\mathbb N}$, where for each 
$n\in\mathbb N$ we have $|\Delta_n| = n-2$, that satisfy the 
following \new{Stasheff identities}
\[\text{SI}(n):\quad
	 \sum_{r+s+t=n} (-1)^{r+st} 
	 	(1^r\otimes \Delta_s \otimes 1^t)\Delta_u = 0. \]
That such sequence of maps be locally finite means that for
each element $v\in V$ the set $\{ \Delta_n(v) : n\in\mathbb N_0\}$ 
contains finitely many nonzero terms. We write $(V,\Delta)$
for an $A_\infty$-coalgebra, and we call it \emph{minimal} whenever $\Delta_1$
vanishes. Observe that every graded vector space, every complex, 
and every dga coalgebra is, in an obvious way, an $A_\infty$-coalgebra. \textcolor{comcol}{Remark our $A_\infty$-coalgebras are non-unital and positively
graded. We warn the
reader that conventions for the signs appearing
in the Stasheff identities above vary in the
literature. As explained in~\cite{Markl}*{Section 2}, one can go from this convention to that
of J. Stasheff by multiplying $\Delta_n$ by the
sign $(-1)^{\binom{n}{2}}$.}

{{}} We can associate to every $A_\infty$-coalgebra $(V,\Delta)$ 
a dga algebra $(\Omega_\infty V,b)$, its \new{$\infty$-cobar construction}, as follows.
The underlying algebra to $\Omega_\infty V$ is the free associative algebra
on the suspension $s^{-1}V$. Define the family of maps 
$(b_n : s^{-1} V
\longrightarrow (s^{-1}V)^{\otimes n})_{n\in\mathbb N}$ by conjugation with the isomorphisms $s : s^{-1}V \longrightarrow V$ and 
$(s^{-1})^{\otimes n} : V^{\otimes n} \longrightarrow (s^{-1}V)^{\otimes n}$. This 
sequence gives a map $s^{-1}V \longrightarrow \Omega_\infty V$, and we then 
have a unique derivation $b : \Omega_\infty V\longrightarrow \Omega_\infty V$ 
that restricts to $s^{-1}V \longrightarrow \Omega_\infty V$ on $s^{-1}V$. 
A straightforward computation shows that $b^2=0$ is equivalent to the Stasheff
identities, so we have a dga algebra. Observe that since $V$ is positively
graded, $\Omega_\infty V$ is non-unital and non-negatively graded. 
If $V$ has a weight grading, as it happens for $\Tor_A$ whenever
$A$ is a weight graded algebra, $\Omega_\infty V$ inherits a
weight-grading from $V$. 

{{}} The $\infty$-cobar construction allows us to define
the category of $A_\infty$-coalgebras, which we denote by
$\DCSH$, quite painlessly: its objects are the 
$A_\infty$-coalgebras and the hom-sets are given by 
\[ \hom_{\DCSH}(C,C')= \hom_{\DA}(\Omega_\infty C,\Omega_\infty C').\] Plainly, $\DCSH$ is the 
full subcategory of $\DA$ consisting of dga algebras that are 
quasi-free, that is, those which are free as graded algebras
if we forget about their differential. Since in the category $\DA$ we have defined the notion of
homotopy between maps and weak equivalences, the quasi-isomorphisms, these notions are available to us in $\DCSH$ by creating them with the functor $\Omega_\infty$; see~\cite{Has}. Observe, moreover, 
that if $F :  V\rightsquigarrow  W$ 
is a map between $A_\infty$-coalgebras,  
it is determined uniquely by a sequence of maps $(f_n : V\longrightarrow 
W^{\otimes n})_{n\in\NN}$ satisfying appropriate commutativity conditions with
the coproducts of $V$ and $W$. In view of this, we will identify such a map
$F$ with the sequence $f=(f_n)_{n\in \NN}$, and write $\Omega_\infty(f)$ for 
$F$. Abusing notation a little, for a second sequence $(g_n)_{n\in\NN}$,
we write $fg$ for the map corresponding to the composition $\Omega_\infty(f)\Omega_\infty(g)$.

{{}} Let $C$ be a dga coalgebra, and assume that $V$ is a complex of 
$\kk$-modules which is a deformation retract of $C$, that is, 
there is a homotopy retract datum
\[\label{diag:HRD}
\begin{tikzcd}
V 
	\arrow[shift left = .5 em]{r}{i} & 
C 
	\arrow[shift left = .5 em]{l}{p}
	\arrow[out=30, in=330,distance=3em, loop]{}{h}
\end{tikzcd},\quad 1-ip = dh+hd,\quad pi =1,
\] 
which we denote by $(i,p,h)$. We assume that such datum 
satisfies the side conditions, that is, all three maps
$h^2$, $hi$ and $ph$ are zero, in which case we call it
a \new{contraction}. The following result, which
is a simplified form of Theorem 5 in \cite{Markl}, shows how
to transfer on $V$ an $A_\infty$-coalgebra structure from
the dga coalgebra structure of $C$ and, further, how to produce 
from a homotopy retract datum another homotopy datum of
$A_\infty$-coalgebras.

\begin{theorem}[Homotopy Transfer Theorem]\label{thm:HTT} 
Let $(C,\Delta_2')$ be a dga coalgebra and consider a homotopy retract as above. There exists an $A_\infty$-coalgebra structure
on $V$ and a homotopy retract datum
\[
\begin{tikzcd}
 \Omega_\infty V 
	\arrow[shift left = .5 em]{r}{j} & 
 \Omega_\infty C 
	\arrow[shift left = .5 em]{l}{q}
	\arrow[out=30, in=330,distance=3em, loop]{}{k}
\end{tikzcd},\quad 1-jq  = bk+kb,\quad qj =1.
\]
The $A_\infty$-coalgebra structure on $V$ is given by $\Delta_1=d_V$ and, for $n\in \NN_{\geqslant 2}$, by $\Delta_n = p^{\otimes n}\Delta_n' i$, where for $n\in\NN_{\geqslant 3}$ the arrows $\Delta_n' : C\longrightarrow C^{\otimes n}$ are defined by
\[ \Delta_n' = \sum_{
	\substack{s+t=n \\ s,t>0}} (-1)^{s(t+1)} 
	(\Delta_s' h\otimes\Delta_t' h)\Delta_2', 
	\]
with the convention that $\Delta_1'h=1$.  \qed
\end{theorem}

\begin{figure}[b]
\centering
\vspace{0 em}
\begin{tikzpicture}[scale = 1.2]
%\draw[black, fill =white] (.5,2.75) circle [radius = .15 cm];
%\draw[black, fill =white] (1.5,2.75) circle [radius = .15 cm];
\draw[black, fill =black] (1,2.25) circle [radius = .05 cm];
\node at (1,2.25) {};
\node at (1.5,2.75) {};
\node at (.5,2.75) {};
\node at (1,2.25) {};
\node at (2,3.25) {};
\node at (0,3.25) {};
\begin{pgfonlayer}{bg}   
\draw (1,1.50)--(1,2.25);
\draw (1,2.25)--(2,3.25);
\draw (1,2.25)--(0,3.25);
\draw (2,3.25)--(3,4.25);
\draw (2,3.25)--(1,4.25);
\draw (3,4.25)--(4,5.25);
\draw (3,4.25)--(2,5.25);
\end{pgfonlayer}
\end{tikzpicture} 
\caption{A right comb}
\end{figure} 

{{}} There is a non-inductive definition of the maps 
$(\Delta_n)_{n\in\mathbb \NN}$ that will be useful to have in mind when 
we discuss $A_\infty$-coalgebra structures on $\Tor_A$, which can also
be found in \cite{Markl}*{Section 4}. Let $T$ be a planar binary tree with $n$ leaves, 
and let us assign to it a sign $\vartheta(T)$ as follows. For each 
vertex $v$ of $T$, let $r_1$ be the number of paths from a leaf of
$T$ to the root that pass through the first (left) input of $v$, and let
$r_2$ be the number of those that pass through the second (right). 
Set $\vartheta_T(v)=
r_1(r_2+1)$ and $\vartheta(T) = \sum_{v\in T} \vartheta_T(v)$. It will be important later on to observe that if $T$ is the right comb with $n$ leaves then $\vartheta(T) =
\binom{n+1}2 -1$.
Let us write $\Delta_T$ for the cooperation of arity $n$
obtained by decorating the leaves of $T$ by $p$, the root of $T$ by $i$, 
the inner vertices by $\Delta_2'$ and the inner edges by $h$; see Figure 1 in~\cite{Markl}*{Section 4}. We then have
the following result.

\begin{theorem}\label{thm:treesdesc}
Let $n\in\mathbb N$. Then $\Delta_n$ is given by
the sum $\sum_T (-1)^{\vartheta(T)}\Delta_T$
as $T$ ranges through all planar binary trees with $n$ leaves. 
\qed 
\end{theorem} 

\section{The \texorpdfstring{$A_\infty$}{Ainfty}-coalgebra structure on Tor}\label{sec2:homcomb}

\subsection{The homotopy retract}\label{ssec:proofs}
 
Let $A$ be an algebra with a reduced Gröbner 
as in the construction of Jöllenbeck--Welker. 
Using the contraction 
\[ 
\begin{tikzcd}
\Tor_A\otimes A
	\arrow[shift left = .5 em]{r}{i} & 
B(A,A) 
	\arrow[shift left = .5 em]{l}{p}
	\arrow[out=30, in=330,distance=3em, loop]{}{h}
\end{tikzcd},\quad 1-ip = dh+hd,\quad pi =1,
\] 
obtained from Subsection~\ref{sec:ADMT} and from the
Morse matching for $B(A,A)$ described in Subsection~\ref{sec:AADMT}
we obtain, upon tensoring to the right with $\kk$, a
contraction
\[ 
\begin{tikzcd}
\Tor_A
	\arrow[shift left = .5 em]{r}{i} & 
BA
	\arrow[shift left = .5 em]{l}{p}
	\arrow[out=30, in=330,distance=3em, loop]{}{h}
\end{tikzcd},\quad 1-ip = dh+hd,\quad pi =1,\]
from the dga coalgebra $BA$ to its homology, $\Tor_A$.
This and the Homotopy Transfer 
Theorem~\ref{thm:HTT} provide us with a minimal $A_\infty$-coalgebra 
structure on $\Tor_A$, and which we will describe in detail 
by means of the combinatorics of Anick chains. It is worthwhile
to note that one may obtain this retract directly, by applying
the methods of~\cite{ADMT} to the bar construction $BA$. 
 
{{}} We recall that by construction $hi=0$, that is, $h$
vanishes on $\Tor_A$. Suppose now that $\gamma = [ x_{i_1} \vv u_1\vv 
\cdots \vv u_r ]$ is a bar term representing an Anick chain in $\Tor_A$.
We then have that $\Delta_2'(\gamma) = \sum \gamma_{(i)}\otimes \gamma^{(i)}$
where each left term $\gamma_{(i)}$ is also a chain: $\gamma_{(i)}$ is
the unique $i$-chain obtained from $\gamma$ by removing a right divisor. 
Since $\Delta_n$ is obtained by projecting the map $\Delta_n' : BA
 \longrightarrow BA^{\otimes n}$, defined recursively in Theorem~\ref{thm:HTT},
 we obtain the following. 

\begin{prop}\label{prop:flushright} For $n\geqslant 3$ we have that
$ \Delta_n' = (-1)^n(1 \otimes \Delta_{n-1}'h)\Delta_2'$
on $\Tor_A$.
\end{prop}
\begin{proof}
Let $\gamma$ be a chain. 
Looking at the recursive definition of higher
coproducts given by~\ref{thm:HTT}, any term which
contains $\Delta_s'h$ on the left for some
$s\geqslant 2$ will act by zero on $\Delta_2'(\gamma)$, 
since the all the terms to the left of the tensor
appearing in this sum are also chains, and
we already know $hi=0$. 
\end{proof}

\subsection{Description of the homotopy}
 
{{}} In this section, we consider only the case when
$A$ is monomial. From the last proposition of the previous section, 
it follows in particular that $\Delta_3' = -(1\otimes \Delta_2' h)\Delta_2' $ on $\Tor_A$,
so the only tree that appears in $\Delta_3$ is the right comb.
Given $n\in\NN_{\geqslant 3}$, we would like to show this is 
the case for the higher coproduct
\[ \Delta_n : \Tor_A\longrightarrow \Tor_A^{\otimes n} .\]
defined by $p^{\otimes n} \Delta_n' i$. We will need an 
explicit description of the homotopy $h$. Because 
it will be useful later on, we also give a description of
the projection $p$: this is the content of the following  
lemma.  Let us say a bar term $[u_0\vv\cdots\vv u_r]$ is
\new{attached} if for $i\in\{1,\ldots,r\}$, we have $u_{i-1}u_i=0$. 

\label{par:handp} 
Suppose that $\gamma =  [u_0\vv\cdots\vv u
_r]$ is attached but is not a chain. Then there is a 
largest $i_1$ such that $u_i = u_i'u_i''$ and such 
that $\eta^1 = [u_0\vv\cdots\vv u_i']$ is a chain. Remark that
by this we mean the \emph{bar structure} is also the correct one;
for example, $[t\vv t^2]$ and $[t^2\vv t]$ both have underlying
monomial the chain $t^3$ in $\kk[t]/(t^3)$ but
only the first is a $1$-chain.  
It may happen that $i=0$, in which case $u_0'$ 
is simply the first variable in $u_0$, as it does for $[t^2\vv t]$. 
We define 
\[\gamma^1 = (-1)^{i_1+1}[\eta^1 \vv {u_{i_1}''}\vv 
	{u_{i_1+1}} \vv \cdots \vv 
	{u_r}],\quad
	\Gamma^1 = [\eta^1 \vv u_{i_1}'' 
	{u_{i_1+1}} \vv \cdots \vv 
	{u_r}]. \]
If $\Gamma^1$ is a chain or zero, stop. Else, there is 
some largest $i_2>i_1$ such that, keeping in with the 
notation above, $\eta^2 = [u_0\vv \cdots \vv u_{i_1}' 
\vv\cdots \vv u_{i_2}'']$ is a chain.
In which case, set
\[\gamma^2 = (-1)^{i_2+1}[\eta^2 \vv
	u_{i_2}'' \vv u_{i_2+1}\vv
	\cdots \vv 
	{u_r}], \quad
	\Gamma^2 = [\eta^2 \vv
	{u_{i_2}''}u_{i_2+1} \vv 
		\cdots\vv 
	{u_r}].\]
Continuing in this way, we obtain terms 
$\gamma=\Gamma^0,\cdots,\Gamma^{n}$ and 
$\gamma^1,\cdots,\gamma^n$, where 
$\Gamma^{n}$ is either zero or a chain. 
For convenience, we will agree that $\gamma^m = 0$ for 
$m>n$, and note that the sign accompanying $\gamma^a$
is $(-1)^{i_a+1}$, where $i_a$ is simply the length
of the largest chain $\eta^a$ contained in $\gamma^a$, starting
from the left.
If $\gamma$ is a bar term in degree $r+1$ whose underlying 
monomial is an $r$-chain, we will write $\Gamma$
for the $r$-chain obtained from $\gamma$ at the end of
the algorithm above, which we observe has no signs. 
Observe that by construction, the sequence 
$(i_a)_{a\geqslant 1}$ is strictly increasing, until it stabilizes. 

\begin{lemma}\label{lemma:handp} With the notation above, 
we have that
\[ h(\gamma) = \sum_{i=1}^n 
 \gamma^i,\quad 
p(\gamma) 
=\Gamma.\]
\end{lemma}

\begin{proof}
Suppose that $\gamma = [u_0\vv\cdots\vv u_r]$ is attached. 
From the
description of the Morse graph in Lemma~\ref{lemma:edges}, we see that
there is a unique inverted edge from $\gamma$ to the element
$\gamma^1$ in the previous paragraph. The face maps of
$\gamma^1$ are all zero or $\gamma$ except possibly for
$\Gamma^1$, up to sign. If $\Gamma^1$ is critical, there is no
inverted edge leaving $\Gamma^1$, and so $h$ is what we claim.
Else, we can repeat the argument above. The claim for $p$ 
follows, \textcolor{comcol}{since there is a unique path from
$\gamma$ to $\Gamma$ in the Morse graph, which
is obtained by following the terms $\gamma^1,\Gamma^1,\gamma^2,\Gamma^2,\ldots$ that we
described above. Finally, 
 the signs can be read off
from the definition of the Morse graph
and the differential of the bar construction.
Indeed, the signs in the bar construction are 
such that when a bar is removed and two terms
are multiplied, there is a sign $(-1)^N$ where
$N$ is the number of terms preceding the first
factor that was multiplied, and the Morse graph
inverts and negates every sign, effectively
changing every $-1$ to a $1$, and vice-versa.
A close look shows that this are precisely the 
signs we have incorporated in our description
of the elements $\gamma^1,\gamma^2,\ldots$ and $\Gamma^1,\Gamma^2,\ldots$, which completes the
proof the Lemma.}
\end{proof}

{{}} \textcolor{comcol}{
It is useful to observe that 
uniqueness follows precisely because our algebra
is monomial, and hence one can either ``undo'' a 
differential in the Morse graph in a unique way,
or fail to do so. If we had more complex 
rewriting rules, it would be a priori possible
to obtain a term in multiple ways from reduced
monomials, causing our paths to branch and
making the description of the action of $h$
and $p$ above much more complicated.} In the language of the Morse graph of $M$, we
have the following corollary.

\begin{cor} Let $c$ be a vertex in $G^M$ of degree
$t$ that is not critical. There is a unique element
$c'$ of degree $t+1$ and a unique element $c''$
of degree $t$, which is either zero or critical,
a unique path in $G^M$ from $c$ to $c'$ and, if
$c''$ is nonzero, a unique edge from $c'$ to $c''$.
Thus, the coefficients in the homotopy of 
Theorem~\ref{thm:mainadmt} are all $1$ or $-1$ and
$p(c)$ coincides with $c''$. 
\end{cor}

\begin{proof}
\textcolor{comcol}{The proof of Lemma~\ref{lemma:handp} shows
there is a unique path to follow when computing
the action of the homotopy $h$ on a non-critical
vertex, and is given by the successive terms
$\gamma^1,\gamma^2,\ldots$ and $\Gamma^1,\Gamma^2,\ldots$ that we wrote down explicitly
before its proof. The conclusion that the coefficients in the homotopy are all $-1$ or
$1$ follow also from the careful description
of the terms above, since we gave the signs 
explicitly. The unique element $c''$ corresponds
to $\Gamma$, while the unique element $c'$
corresponds to the last non-zero term in
the sequence $(\gamma^1,\gamma^2,\ldots)$.}
\end{proof}

{{}} To illustrate, let us consider a
third algebra $L= T(t)/(t^N)$ where $N>2$. We then have
\[ h[\overbrace{t^{N-1}\vv t}^{m}] = -\sum_{i=0}^m 
	[\overbrace{t\vv t^{N-1}}^{i} \vv t\vv t^{N-2} \vv 
		\overbrace{t\vv t^{N-1}}^{m-i-1}\vv t]\]
		%\text{, and}\] 
%\[ h[\overbracket{t^{N-1}\vv t}^{m}\vv t^{N-1}] = \sum_{i=0}^m 
%		[\overbracket{t\vv t^{N-1}}^{i} \vv t\vv t^{N-2} \vv 
%		\overbracket{t\vv t^{N-1}}^{m-i}].\]
where the brackets mean the terms are repeated the indicated
amount of times. Note that, since in every summand the homotopy
extracted a chain of odd homological degree, all the signs are
the same.
Using the results of Section~\ref{sec3:mainres} 
the reader may recover the $A_\infty$-coalgebra structure 
on $\Tor_A$ for $p$-Koszul monomial algebras, dual to the 
$A_\infty$-algebra structure on $\Ext_A$ obtained in~\cite{HeLu}.

\subsection{The exchange rule and the right comb}

{{}} We now prove the desired result that when computing the higher
coproducts in $\Tor_A$ obtained from the homotopy retraction 
datum of Section~\ref{sec:AADMT}, the only contributing tree is 
the right comb. The following exchange rule for $h$ and $\Delta_2'$ will easily imply this result. 

\begin{lemma}\label{lemma:exchange}
If $\gamma$ is attached then
$\Delta_2'(h(\gamma)) = (h\otimes 1) 
\Delta_2'(\gamma)$ modulo $\Tor_A \otimes BA$.
\end{lemma}
%\noindent We can depict this situation 
%by the following exchange rule:
%\[
%\begin{tikzpicture}
%\node at (1,0.75) {$h$};
%\node at (1,2) {$\bullet$};
%\node at (2,3) {};
%\node at (0,3) {};
%\draw (1,1)--(1,2);
%\draw (1,2)--(2,3);
%\draw (1,2)--(0,3);
%\end{tikzpicture}
%\begin{tikzpicture}
%\node at (1,0) {};
%\node at (1,1) {$=$};
%\end{tikzpicture}
%\begin{tikzpicture}
%\node at (0,1) {$\hphantom{x}$};
%\node at (0,3.5) {$h$};
%\node at (1,2.25) {$\bullet$};
%\node at (2,3.25) {};
%\node at (0,3.25) {};
%\draw (1,1)--(1,2.25);
%\draw (1,2.25)--(2,3.25);
%\draw (1,2.25)--(0,3.25);
%\end{tikzpicture} 
%\]
%\noindent Remark that if $\gamma$ is critical the 
%above is trivially true, since both sides are $0$. 
\begin{proof}
This is a direct computation, albeit a 
bit cumbersome. We will use the notation of \ref{par:handp}.
Let $\gamma= [u_0\vv\cdots\vv u_r]$ and write 
$\Delta_2'(\gamma) = \sum_{i=1}^r \gamma_{(i)}\otimes \gamma^{(i)}$. 
From the definitions it follows that if $j\in \{1,\ldots,r\}$
then:
\begin{tenumerate}
\item $(\gamma_{(j)})^a=0$ if $j< i_a$.
\item $(\gamma^a)_{(j)}$ is a chain for $j\leqslant i_a+1$.
\item $(\gamma^a)_{(j)} = (\gamma_{(j-1)})^a$ for
$j\geqslant i_a +1$. 
\item $(\gamma^a)^{(j)} = \gamma^{(j-1)}$ for
$j > i_a +1$. 
\end{tenumerate}
\textcolor{comcol}{
That $(1)$ holds follows, for if the chain
we want to extract from $\gamma$ appears after the $j$th bar, then $\gamma_{(j)}$ will be
too short to contain it. It is clear that $(2)$ holds, since $\gamma^a$
has been ``straightened'' to chain structure up 
to the $i_a$th bar, and we already observed any initial
bar term of a chain is again a chain. Note that
$(3)$ says that we can either straighten $\gamma$
to a chain up to step $a$ and then truncate
far from where this chain ends, or we can 
truncate $\gamma$ at worst at the boundary 
and then straighten it to a chain: the result is
the same. Finally, $(4)$ says that the tail
of $\gamma$ is not affected if cut beyond the
straightening bar.} This means that we can write
\begin{align*}
\Delta_2'(h(\gamma)) &= 
	\sum_{a\geqslant 1} \sum_{j\leqslant r+1}
		(\gamma^a)_{(j)}\otimes (\gamma^a)^{(j)}\\
	&= \sum_{a\geqslant 1} \sum_{j\leqslant i_a+1}
		(\gamma^a)_{(j)}\otimes (\gamma^a)^{(j)}+
	\sum_{a\geqslant 1} \sum_{i_a < j-1 \leqslant r}
		(\gamma^a)_{(j)}\otimes (\gamma^a)^{(j)}\\
		\end{align*}
		\begin{align*}
\hphantom{\Delta_2'(h(\gamma)) }	&= \sum_{a\geqslant 1} \sum_{j\leqslant i_a+1}
		(\gamma^a)_{(j)}\otimes (\gamma^a)^{(j)}+
	\sum_{a\geqslant 1} \sum_{i_a < j-1 \leqslant r}
		(\gamma_{(j-1)})^a\otimes \gamma^{(j-1)} \\
		&= \sum_{a\geqslant 1} \sum_{j\leqslant i_a+1}
		(\gamma^a)_{(j)}\otimes (\gamma^a)^{(j)}+
	\sum_{a\geqslant 1}\sum_{i_a \leqslant j \leqslant r}
		(\gamma_{(j)})^a\otimes \gamma^{(j)}
\end{align*}
where the third equality uses (iii) and (iv), and from
(ii) it follows the first summand is in $\Tor_A\otimes BA$.
Finally, from (i) it follows that the second sum
is, in fact, $(h\otimes 1)(\Delta_2'(\gamma))$, which
completes the proof of the lemma.
\end{proof}

\begin{cor} We have $(h\otimes 1)\Delta_2' h = 0$ on
attached bar terms. 
\end{cor} 

\begin{proof}
This now follows from our exchange rule and the fact $h$ has square zero  and vanishes on $\Tor_A$.
\end{proof}

\begin{theorem}\label{thm:comb}
Let $n\in\NN_{\geqslant 3}$ and let $\gamma\in \Tor_A$ be an element
represented by an Anick chain. The only tree that contributes to 
$\Delta_n'(\gamma)$, and hence to $\Delta_n(\gamma)$, is the right comb.   
\end{theorem}
\begin{proof}
The
fact that $h$ vanishes on $\Tor_A$ means that, at 
the root, the left edge must be a leaf. Knowing this,
the exchange rule means that if $T$ is planar and contains any subtree of the form
$$
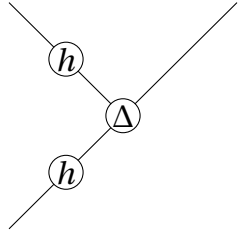
\begin{tikzpicture}[scale = 1.5]
\draw[black, fill =white] (.5,2.75) circle [radius = .15 cm];
\draw[black, fill =white] (.5,1.75) circle [radius = .15 cm];
\draw[black, fill =white] (1,2.25) circle [radius = .15 cm];
\node at (1,2.25) {\small $\Delta$};
\node at (.5,1.75) {\small $ h$};
\node at (.5,2.75) {\small $ h$};
\node at (1,2.25) {};
\node at (2,3.25) {};
\node at (0,3.25) {};
\begin{pgfonlayer}{bg}   
\draw (0,1.25)--(1,2.25);
\draw (1,2.25)--(2,3.25);
\draw (1,2.25)--(0,3.25);
    \end{pgfonlayer}
\end{tikzpicture} 
$$
which corresponds to $(h\otimes 1)(\Delta_2'h)$,
the operator $\Delta_T$ will vanish identically. 
This means that the only tree that may possibly give
a nonzero contribution to $\Delta_n$ is the right
comb.
\end{proof}

{{}} Let us also record here the following easy 
proposition, which means, plainly, that the computation of 
the $A_\infty$-structure of $\Tor_A$ depends only on the
local information on a given chain. Thus, there seems to be 
no upshot from looking at induced maps when relations are added.

\begin{prop}
Suppose $A$ is a monomial algebra and $B$ is obtained by adjoining
to $A$ a non-redundant monomial relation. Let $\varphi : A\to B$ be the
quotient map. Then the map $\Tor_\varphi : \Tor_A\longrightarrow \Tor_B$
identifies $\Tor_A$ as a sub-$A_\infty$-coalgebra of $\,\Tor_B$ in such
a way that the coproducts of $\,\Tor_A$ are the restriction of those of $\,\Tor_B$ through $\Tor_\varphi$.  
\end{prop}
\begin{proof}
\textcolor{comcol}{
Since $B$ is obtained from $A$ by adjoining a 
non-redundant monomial relation, the collection
of Anick chains for $B$ can be computed from
those of $A$ by adding a (possibly infinite) 
new collection of chains, and the map $\Tor_\varphi$ is injective since it is induced by
the inclusion map on Anick chains. To see that
this a strict map of $A_\infty$-coalgebras,
meaning that it induces on $\Tor_A$ the correct
higher coproducts, we note that we can arrange it
so that the contraction on the bar complex of $B$
onto $\Tor_A$ restricts to a contraction
for the bar complex of $A$: according to the
work of Jollenbeck--Welker, this datum can be
produced exclusively from the Anick chains
from $B$, and their procedure does not alter
the underlying monomial of a chain in the 
monomial case, and hence restricts to the 
bar complex of $A$. 
Finally, the higher coproducts are built from
the coproduct of the bar construction of $B$
and the contraction, and for each Anick chain,
this computation depends only on the underlying
chain, and not on the inclusion of $A$ into $B$.
} 
\end{proof}

\section{Description of the minimal model}\label{sec3:mainres}

We now aim to give a more refined description
of the terms appearing in a higher coproduct of a 
fixed chain $\gamma$, as stated in the following theorem.
It will follow immediately from Theorem~\ref{thm:MRmin} and
its proof. Unless stated otherwise, we are
working exclusively with monomial algebras in
what follows. 

\begin{theorem}\label{thm:MRcop}
Let $\gamma$ be a chain and $n\in\NN_{\geqslant 2}$. The terms that 
appear in $\Delta_n(\gamma)$ are exactly those of the form 
$\gamma_1\otimes\cdots\otimes\gamma_n$ with $(\gamma_1,\ldots,\gamma_n)$ a 
decomposition of $\gamma$. Moreover, if $\gamma_i$ is of
length $r_i$ for each $i\in \{1,\ldots,n\}$, the coefficient of 
$\gamma_1\otimes\cdots\gamma_n$ is $(-1)^N$ where
\[
 N = \binom{n+1}{2} +r_1+
 	\sum_{i=1}^{n-1} (n-i)(r_i+1).  \]  
\end{theorem}
\subsection{Combinatorics of chains and tails}

{{}}  Suppose that $\gamma = x_{i_1}\cdots x_{i_s}$ is an Anick 
chain, with associated interlaced sequences $\{(a_j),(b_j)\}$. 
We will say a variable $x_{i_s}$ is an overlapping variable if
$s\in [a_{j+1},b_j)$, and we will say that 
a bar is \new{inserted at $x_{i_s}$} if it is inserted immediately 
after it. A bar term obtained from $\gamma$ is \new{regular} if it is 
obtained by inserting bars at non-overlapping variables, and it is 
\new{coregular} if it is obtained by inserting bars at overlapping 
variables. It may happen that $a_{j+1}=b_j$, in which case we
agree that $x_{i_{a_{j+1}}}$ is both overlapping and non-overlapping. 
This always happens, for example, if $A$ is quadratic. 
The following figure illustrates our definitions for the $4$-chain
$[t\vv t^3\vv t\vv t^3\vv t]$ in $\kk\langle t\vv t^4\rangle$,
where white circles represent overlapping variables,  black ones
represent non-overlapping variables, the cross represents 
the only variable that is both overlapping and non-overlapping, 
and bars mark the obstructions that constitute the chain.

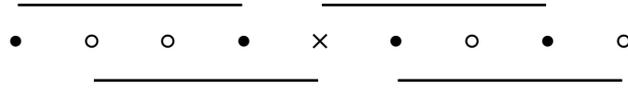
\vspace{1 em}
\begin{figure}[h]
\centering
\begin{tikzpicture}
\foreach \i in {1,4,6,8}
{
\node (\i) at (\i,0) {$\bullet$};
		}
\foreach \i in {2,3,7,9}
{
\node (\i) at (\i,0) {$\circ$};
		}
\node (5) at (5,0) {$\times$};
\node[inner sep=0pt] (m1) at (1,.5) {};
\node[inner sep=0pt] (m2) at (2,-.5) {};
\node[inner sep=0pt] (m3) at (4,.5) {};
\node[inner sep=0pt] (m4) at (5,.5) {};
\node[inner sep=0pt] (m4') at (5,-.5) {};
\node[inner sep=0pt] (m5) at (6,-.5) {};
\node[inner sep=0pt] (m6) at (8,.5) {};
\node[inner sep=0pt] (m7) at (9,-.5) {};
\draw[line width = 1pt] (m1) -- (m3);
\draw[line width = 1pt] (m2) -- (m4');
\draw[line width = 1pt] (m4) -- (m6);
\draw[line width = 1pt] (m5) -- (m7);
\end{tikzpicture}
\caption{Schematics of the chain $t^9$ in $\kk\langle t\vert t^4\rangle$.}
\end{figure}

\begin{lemma}\label{lemma:bounds}
Let $\gamma$ be a monomial which is an $r$-chain. Any (co)regular bar term obtained by inserting
\begin{tenumerate}
\item exactly $r$ bars into $\gamma$ is either attached and nonzero or is zero,
\item less than $r$ bars into $\gamma$ is zero, and
\item more than $r$ bars into $\gamma$ is not attached and nonzero or is zero.
\end{tenumerate}
\end{lemma}
\begin{proof}
We prove this by induction on $r$. If $r=1$, then $\gamma$
is simply a monomial relation. Certainly inserting no bars
gives a bar term of degree one which is zero and, since there
are no overlapping variables to keep track of, inserting any
bar gives a regular bar term, which is certainly nonzero, and
inserting one more bar gives a non-attached term. Assume then
$r\geqslant 1$ and that our claim holds for $r$-chains, and
that we have an $(r+1)$-chain. We consider the three cases above
separately:
\begin{tenumerate}
\item \emph{We have inserted $r+1$ bars regularly:} if the bar
term is zero, we are done. Else the bar term obtained in 
nonzero, and there must be at least one bar inserted in a 
non-overlapping variable of the last chain. Moreover, there
must be exactly one, else, by removing the tail of the $r+1$
chain, we would obtain a regular bar term from an $r$-chain
which is nonzero but has $r-1$ bars, which cannot happen. Having
settled this, we now remove the tail and proceed by induction.
\item \emph{We have inserted less than $r+1$ bars regularly:}
if no bar has been inserted on non-overlapping variables of
the last monomial relation, we are done. Else, there is one
variable inserted there. Removing the tail now gives a regular
bar term obtained from an $r$-chain were less than $r$ bars
have been inserted, and induction does the rest.
\item \emph{We have inserted more than $r+1$ bars regularly:}
if two or more bars have been inserted in non-overlapping 
variables of the last monomial relation, we get a zero term, 
since removing the tail gives a bar term where at most $r-1$ 
bars have been inserted regularly into an $r$-chain. If there
is exactly one bar in the tail, we may remove it and proceed
inductively.  
\end{tenumerate}
Analogous considerations apply to coregular terms.
\end{proof}

{{}} We now note that the homotopy $h$, which introduces
and shifts bars in bar terms, produces bar terms whose subchains, 
starting from the left, have bars introduced regularly.

\begin{lemma}
If $\gamma$ is an element of $\Tor_A^{r+1}$ corresponding to an 
$r$-chain, it has its $r$ bars inserted regularly. In particular,
if $\gamma$ is an attached term, and if $\gamma^a$ is a
nonzero summand in $h(\gamma)$, following the notation of
Lemma~\ref{lemma:handp}, then for $j\leqslant i_a$, the $j$-chain 
$(\gamma^a)_{(j+1)}$ has its $j$ bars inserted regularly.
\end{lemma}

\begin{proof}
The insertion of bars follows
Anick's interlaced sequence associated to a chain in such a
way that we inserts bars at variables $x_{i_1},x_{i_{b_1}}, 
\ldots,x_{i_{b_{r-1}}}$ which are not overlapping, since the
overlapping variables are precisely at the half-open intervals 
$[{a_j},{b_{j-1}})$ for $j\in \{2,\ldots,r-1\}$. 
\end{proof}

{{}} Let us now introduce the definitions that will be central to our proof of 
Theorems~\ref{thm:MRcop} and its equivalent formulation, 
Theorem~\ref{thm:MRmin}, which we already stated the Introduction. 
Let $\gamma$ be an $r$-chain and $j\in\NN$.
 We will say a bar term $\Gamma$ is a \new{$j$-tail of
$\gamma$} if there is a term of the form $\gamma_1\otimes
\cdots \otimes \gamma_{j}\otimes \Gamma$ in $\Delta_{j+1}'(\gamma)$
appearing with nonzero coefficient, where the first $j$
tensors are chains, and, moreover, $\Gamma$ is a concatenation
of at least two chains $\gamma_{j+1},\ldots,\gamma_n$, in this
order. Moreover, if for $i\in [n]$ we have that $\gamma_i$
is an $r_i$ chain, we require that $r_1+\cdots+r_n=r-1$. The 
\new{length} of $\Gamma$ is $n-j$.  Let us call the $n$-tuple
$(\gamma_1,\ldots,\gamma_n)$ a \new{decomposition} of $\gamma$.
Remark that there is the notion of ``tail'' of a chain given in 
\cite{Anick}, but that this is not a special case of our definition, and that $\Gamma$
may be a tail for several choices of the
tuple $(\gamma_{j+1},\ldots,\gamma_n)$. 

{{}} We continue by observing that $j$-tails are 
obtained by cutting a chain in the form of a bar term
either at a bar or at some place between bars. 

\begin{lemma}\label{lemma:shape}
Fix $j\in\NN$ and suppose that $\gamma = [u_0\vv u_1\vv
\cdots\vv u_r]$ is an $r$-chain, and that $\Gamma$ is a
$j$-tail of $\gamma$, with first chain $\gamma_{j+1}$. 
Then there exists $i\in \{1,\ldots,r\}$ and a decomposition
$u_i = u_i'u_i''$ such that $u_i''\neq 1$, $u_i''u_{i+1}=0$
minimally and $\Gamma = [u_i''\vv\cdots\vv u_r]$. Moreover:
\begin{tenumerate}
\item This decomposition is nontrivial whenever $j>1$
\item The tail $\Gamma$ contains exactly $r_{j+1}+\cdots + r_n$ bars.
\item There is a unique $(j-1)$-tail $\Gamma'$ and a unique term in
$\Delta_2'h(\Gamma')$ of the form $\gamma_j\otimes \Gamma$ that 
gives rise to $\Gamma$, and it appears with a sign as a coefficient. 
\end{tenumerate}
\end{lemma} 

\begin{proof}
The case that $j=1$ and $n$ is arbitrary is obvious, so let us assume 
$j>1$, our claim true for $(j-1)$-tails, and analyse the claim for $j$. 

Observe that by Theorem~\ref{thm:comb}, 
if $\Gamma$ is a $j$-tail of $\gamma$, it must come from a $(j-1)$-tail
$\Gamma'$ of $\gamma$ by applying the operator $\Delta_2'h$ on the last
factor. We will prove that $\Gamma$ has the desired form and, moreover,
that there is a unique way to obtain $\Gamma$ from $\Gamma'$, so that
if the term corresponding to $\Gamma'$ appears with coefficient $1$ or 
$-1$, then so does the term corresponding to $\Gamma$.
 The description of $h$ from Lemma~\ref{lemma:handp} and the 
inductive hypothesis applied to $\Gamma'$ means that $\Gamma' = 
[u_i''\vv u_{i+1} \vv\cdots\vv u_r]$ with $u_i''u_{i+1}=0$ minimally, 
or $\Gamma'$ has no bars. In the latter case the chain $\gamma_{j+2}$ 
is a variable and then $\Gamma$ is obtained by removing this variable:
we have that $h(\Gamma') = [\gamma_{j+1}\mid\cdots ]$ and we obtain
$\Gamma$ uniquely from $\Gamma'$. 

Let us then consider the case $\Gamma'$ has bars, so
that it
contains $r-1-(r_1+\cdots+r_j)$ bars by induction, and its first
terms overlap minimally. We can certainly find some $k>i$ and a 
decomposition $u_k=u_k'u_k''$ in such a way that the underlying 
monomial of the bar term $[u_i''\vv u_{i+1}\vv \cdots \vv u_{k-1}
\vv u_k']$ is precisely $\gamma_j$. Observe that the bar 
structure of $\gamma_j$ is coregular, for it interlaces with that of $\gamma$. 
By Lemma~\ref{lemma:bounds} there are exactly $r_j$ bars in such term, so that $k = i+r_j$. We now analyse two cases.

\emph{Case 1: $u_i''$ is a variable.} In such case, it follows that 
$[u_i''\vv u_{i+1}\vv \cdots \vv u_{i+r_j}']$ is a honest chain
belonging to $\Tor_A$. We claim that the decomposition $u_k=u_k'u_k''$
is non-trivial, that $u_k''u_{k+1}=0$ minimally and that $h(\Gamma')$ 
is, up to signs, equal to the bar term 
\[ [u_i''\vv u_{i+1}\vv \cdots \vv  u_{i+r_j}'
\vv u_k'' \vv u_{k+1} \vv \cdots \vv u_r]\] which means, of course,
that the description of $\Gamma$ is the correct one. 
Indeed, note that if the decomposition were trivial,
we would have a sequence of chains $\gamma_{j+1}\cdots\gamma_n$ 
underlying a $(j+1)$-tail with less than $r_{j+1}+\cdots+r_n$ bars.
As before, the starting chain $\gamma_{j+1}$ appears with bars
inserted coregularly, so we may remove it along with exactly $r_{j+1}$
bars. Repeating this argument, we end up with a coregular bar term  
underlying an $r_n$-chain with less than $r_n$ bars, which contradicts
Lemma~\ref{lemma:bounds}. 
To see that $u_k''u_{k+1}=0$, note that otherwise we again
would have a bar term $[u_k''u_{k+1}\vv\cdots \vv u_r]$
whose underlying monomial has $r_{j+1}+\cdots+r_n-1$ bars.
The fact that the overlap $u_k''u_{k+1}$ is minimal 
follows from the fact that overlap $u_ku_{k+1}$ is minimal. 
The description of $h(\Gamma')$ shows
that $\Gamma$ is obtained uniquely from $\Gamma'$,
possibly with a sign. 

\emph{Case 2: $u_i''$ is not a variable.} 
Arguing as before, we see that the overlap $u_k = u_k'u_k''$ 
is not trivial, and that $u_k''u_{k+1}=0$ minimally. 
We can write $u_i'' = xv$ were $x$ is a variable and $v$ a 
monomial, and we have that $h(\Gamma')$ has first term 
$[x\vv v\vv u_{i+1} \vv\cdots\vv u_r]$. Set $j_*$ to be the
last $k \geqslant j$ for which $\gamma_k$ is a $0$-chain.
Our bar counting argument then shows that the concatenation
$\gamma_j \cdots \gamma_{j_*}$ must be contained in the
monomial $u_i''$, and then using our description of $h$
it is clear we may extract the term $\gamma_j \otimes \cdots 
\otimes \gamma_{j_*}$ uniquely by iteration of $\Delta_2 h$.
Let us assume then that $\gamma_j$ is not a $0$-chain.
In such case, $vu_{i+1}\neq 0$, since
$\gamma_j$ begins with a minimal monomial relation which,
by minimality, must in fact be $xvu_{i+1}$. 
Since $u_i''u_{i+1}$ is a minimal monomial relation of $A$, it 
follows that $[x\vv vu_{i+1}\vv\cdots\vv u_k']$ begins with
the initial $1$-chain from $\gamma_j$, so that if $\gamma_j$
is a $1$-chain, we are done: this term is of the form 
$[x\vv v u_k']$. Else, we can find the initial $2$-chain
of $\gamma$ in the form $[x\vv vu_{i+1} \vv u_{i+2}']$:
since $u_{i+1}u_{i+2}$ is a minimal monomial relation of 
$A$, the second monomial relation of $\gamma_j$ must be 
contained in a monomial of the form $vu_{i+1}u_{i+1}'$ where
$u_{i+1}'$ is a proper initial divisor of $u_{i+1}$. 
Continuing this way, we see the Anick structure of $\gamma_j$
is interlaced inside that of $\gamma$, and that the last term
in $h(\Gamma')$ is $[x\vv vu_{i+1}\vv\cdots\vv
u_k'\vv u_k''\vv u_{k+1} \vv \cdots\vv u_r]$, proving the 
description of $\Gamma$ is the correct one. 

We also observe 
that the summands of $h(\Gamma')$ different from this one
cannot create a summand corresponding to $\Gamma$ 
so that again $\Gamma$ is 
obtained uniquely from $\Gamma'$. Indeed, the only way to produce
a bar term in the left factor with the same underlying monomial
as $\gamma_j$, we would have to use $\Delta_2'$ to break such a 
term of $h(\Gamma')$ precisely at the bar dividing $u_k'$ and $u_k''$, 
presently only on the last term. If we do it at a bar before or 
after this one, the resulting term has either its left factor or 
its right factor non-attached, since it contains $[\cdots\vv u_k'
\vv u_k''\vv\cdots]$. This same argument shows that the previous
summands of $\Delta_2'h(\Gamma')$ cannot contribute to 
$\Delta_{j+1}$: the only place where we may break them is at the
last opened bar, say $[\cdots\vv u_t'\vv u_t''\vv \cdots]$, but the
fact we can continue the algorithm of \ref{par:handp} means that
$u_t''$ has nonzero product with $u_{t+1}$, and hence this
term does not contribute to $\Delta_{j+1}$. 

The final claim regarding the number of bars in $\Gamma$
is immediate from the above.
\end{proof}
 
\noindent The following proposition is the central result about
tails and chains we were after.  
 
\begin{prop}\label{prop:vewyimp}
Let $\gamma$ be a chain, $n\in\NN_{\geqslant 2}$ and let
$(\gamma_1,\ldots,\gamma_n)$ be a decomposition of $\gamma$. 
For each $j\in [n-1]$ there is a unique $j$-tail $\Gamma$ of 
$\gamma$ with underlying monomial $\gamma_{j+1}\cdots \gamma_n$
and a unique term $\gamma_1\otimes\cdots\otimes \gamma_j\otimes \Gamma$ 
in $\Delta_{j+1}'(\gamma)$, and it appears with coefficient $1$ or
$-1$.
\end{prop} 
 
\begin{proof}
Let $(\gamma_1,\ldots,\gamma_n)$ be a decomposition of $\gamma$
and let $\Gamma$ be a $j$-tail as in the statement of the Theorem.
The claim is obvious for $j=1$. Moreover, Lemma~\ref{lemma:shape} 
shows that once we know that a $(j-1)$-tail $\Gamma'$ corresponding
to this decomposition of appears in $\Delta_j'(\gamma)$, there is a
unique summand in $\Delta_2' h(\Gamma')$, with coefficient $1$ or
$-1$, that produces the term corresponding to $\Gamma$, which is
what we wanted. 
\end{proof}

{{}} Remark that the operators $(\Delta_j')_{j\geqslant 2}$
produce other terms than the ones described in the last proposition.
However, the proof of Lemma~\ref{lemma:shape} shows these terms 
have zero projection to tensor powers of $\Tor_A$, since they contain 
factors that are not attached.

\subsection{Main theorem} 

{{}} We now recall the promised description of the minimal model 
of a monomial algebra $A$. It follows immediately from Proposition
\ref{prop:vewyimp} and Lemma~\ref{lemma:handp}, which in particular
describes the signs appearing in the homotopy $h$. 

\begin{theorem}\label{thm:MRmin}
For each monomial algebra $A$ there is a minimal model
$B\longrightarrow A$ where $B= \Omega_\infty\!
\Tor_A$, and for a chain $\gamma\in\Tor_A$ the differential $d$ acts by
\[ d\gamma  = -\sum_{n\geqslant 2}
	(-1)^{\binom{n+1}{2}+|\gamma_1|} 
	 	 \gamma_1 \cdots \gamma_n, 
	 		\]
where the sum ranges through all possible decompositions of $\gamma$.
\end{theorem}

\begin{proof}
We need only address the claim about signs and the differential
$b$. We already know
that whenever $\Delta_2h$ extracts an $r$-chain, it produces
a sign $(-1)^{r+1}$. Moreover, whenever $h$ goes through
an $r$-chain $\gamma$ it produces a sign $(-1)^{r+1}$. Thus when
creating the term $\gamma_1\otimes\cdots\otimes \gamma_n$ by extracting
$\gamma_{n-1}$, we have a sign $(-1)^L$ where $L = \sum_{i=1}^{n-1}(r_i+1)$. Inductively accounting for the signs created by 
$\Delta_3,\ldots,\Delta_{n-1}$, for the missing sign $r_1+1$
that is \emph{not} created by $\Delta_2$ and for the sign
given by \ref{thm:treesdesc}, we obtain a sign congruent to
\[ \binom{n+1}{2} +r_1 + \sum_{i=1}^{n-1} (n-i)(r_i+1)\pmod{2} ,\]
 which is the integer $N$ in Theorem~\ref{thm:MRcop}. To see the claim 
 about the minimal model, we observe that 
 $ (s^{-1})^{\otimes n} (\gamma_1\otimes\cdots\otimes\gamma_n) = 
  (-1)^M s^{-1}\gamma_1\otimes\cdots\otimes s^{-1}\gamma_n$
  where $M$ is $\sum_{i=1}^{n-1} (n-i)(r_i+1)$, giving the final
  result. 
\end{proof}

{{}}  The canonical identification of $\Ext_A := \Ext_A(\kk,\kk)$
as $\Tor_A^\vee$ gives us a result dual to Theorem~\ref{thm:MRcop}
about the $A_\infty$-algebra structure on $\Ext_A$. Remark that it is quite
crucial to have done all the work with $A_\infty$-coalgebras and then
dualizing to $A_\infty$-algebras, and not otherwise,
since not every $A_\infty$-algebra is dualizable; see \cite{BriBen}*{\S 2.2}.
It is important, however, to pay attention to the Koszul signs arising from
the natural maps $D^n : \Ext_A^{\otimes n}\longrightarrow (\Tor_A^{\otimes n})^\vee$
for $n\in\NN$: if $f_1\otimes\cdots\otimes f_n$ is an element in the domain, and if
we pick $c_1\otimes\cdots\otimes c_n\in \Tor_A^{\otimes n}$, then
\[ D^n(f_1\otimes\cdots\otimes f_n)(c_1\otimes\cdots\otimes c_n) =
			(-1)^N f_1(c_1)\otimes\cdots\otimes f_n(c_n) \]
where $N = \sum_{i=2}^n (|c_1|+\cdots+|c_{i-1}|)|f_i|$. Observe that if $f: 
V\longrightarrow W$ is a map between complexes, then $f^\vee(\varphi) = 
(-1)^{|f||\varphi|} \varphi f$, which explains the introduction of signs
in the higher products of the graded dual $\Ext_A$ of $\Tor_A$. Concretely, for each $n\in\NN_{\geqslant 2}$, define 
$\mu_n : \Ext_A^{\otimes n} \longrightarrow \Ext_A$ by
$ \mu_n(\varphi_1\otimes \cdots \otimes\varphi_n) =
	(-1)^{n(|\varphi_1|+ \cdots +|\varphi_n|)}
	D^n(\varphi_1\otimes \cdots \otimes\varphi_n)\Delta_n$.
Let us say an $A_\infty$-algebra structure on $\Ext_A$ is \emph{canonical}
if it is $A_\infty$-quasi-isomorphic to the dga algebra $BA^\vee$. We
have the following result.

\begin{theorem}\label{thm:Ext}
There is a canonical $A_\infty$-algebra structure on $\Ext_A$ given as follows.
If $n\in\NN_{\geqslant 2}$ and if $\gamma_1^\vee,\ldots,\gamma_n^\vee$ 
are chains in $\Ext_A$ of lengths $r_1,\ldots,r_n$, respectively, then 
$\mu_n(\gamma_1^\vee \otimes\cdots\otimes\gamma_n^\vee)=(-1)^M\gamma^\vee$ 
if the concatenation $\gamma = \gamma_1\cdots\gamma_n$ is a chain of length 
$r=r_1+\cdots+r_n+1$ where $M$ is the integer
$\binom{n+1}2-1+ \sum_{i<j} r_i(r_j+1) + r_1+r$. Otherwise,
this higher product is zero. \qed
\end{theorem}
{{}} 
\textcolor{comcol}{Let us make the model explicit for
the monomial algebra $K$, we will address the 
algebra $J$ later (but note that its associated
monomial algebra has the same Anick chains
as $J$). The dg
model for $K$ is a free algebra $T(x,y,z,\alpha,\beta,\Gamma,\Lambda)$ where the first three
generators are in homological degree $0$, the 
next two in homological degree $1$ and the last 
two in homological degree $2$. The differentials
are as follows:
\[ d\alpha = xy^2,\quad  d\beta = y^2x ,\quad d\Gamma = x\beta -\alpha z, \quad d\Lambda = xy\beta - \alpha yz. \]
These can be read off the (unique) 3-decompositions of the relations into a 
concatenation of variables (0-chains), the
the 2-decompositions of $xy^2z$ into $x\cdot y^2z$ and $xy^2\cdot z$, and the 3-decompositions of
$xy^3z$ into $x\cdot y\cdot y^2z$ and $xy^2\cdot y\cdot z$.
}
\subsection{The extension to monomial quiver algebras}

We now observe that the results of these notes extend 
without any non-trivial modification to the more general class 
of monomial quiver algebras.

{{}} Fix a quiver $Q = (Q_0,Q_1,s,t)$ and 
a set $R$ of paths in $Q$ of length at least two,
none of which is a divisor of another. We call
$A = \kk Q/(R)$ a \new{monomial quiver algebra}. Let us
write $\uk$ for the semi-simple $\kk$-algebra $\kk Q_0$,
so that there is an augmentation $A\longrightarrow \uk$.
We set $\Tor_A = \Tor_A(\uk,\uk)$, and write $BA$
for the bar construction of $A$, where unadorned $\otimes$ are 
now taken over $\uk$. Thus, a generic basis element of $BA$
in degree $n\in\NN$ is of the form $[a_1\vv\cdots\vv a_n]$
where $t(a_i) = s(a_{i+1})$ for each $i\in \{1,\ldots,n-1\}$. Since $\uk$ is semi-simple 
over $\kk$, we can consider this alternative
bar construction instead.

{{}} The methods of Subsection 2.3 go through to produce a homotopy 
retract datum from $BA$ to $\Tor_A$, and select a basis
of $\Tor_A$ of critical vertices given by chains: $\Tor_A^1$ has
basis $\{ [a] : a\in Q_1\}$, and for $n\in\NN$
a basis of $\Tor_A^{n+1}$ is given by bar terms 
$[u_0\vv\cdots\vv u_n]$ where $t(u_i)=s(u_{i+1})$ and
$u_iu_{i+1}=0$ minimally for each $i\in \{0,\ldots,n-1\}$. The
description of the action of the homotopy on fully attached terms
is unchanged, as is the exchange rule. 

{{}} The notion of decompositions of a chain carry through
to this setting, as well as the technical work of Section 3.
As an end result we obtain the following 
description of a minimal model for monomial 
quiver algebras: \textcolor{comcol}{there is no 
change on the Morse matching, the action of the 
homotopy and the exchange rule, or any other 
detail: the fact that we can do things relative 
to $\uk$ makes any pathology that can arise in 
$\kk Q$ due to non-concatenable arrows disappear, 
since elements of the bar construction just look 
like those in the usual one over $\kk$, but with 
the extra condition of concatenability. This last 
condition ensures that spurious cycles, say of 
the form $[x|y]$, arising from zero 
multiplication due $x$ and $y$ not being 
concatenable, disappear, so everything works 
like in the case $Q$ is 
a bouquet. We refer the reader to~\cite{Cibils}*{Lemma 2.1} where it is shown the relative
double sided bar construction is a projective
resolution of $A=\kk Q/(R)$ as an $A$-bimodule (note the hypothesis that $A$ be finite
dimensional is not really needed there).}

{{}} Let us remark that we also have, implicitly,
obtained comparison maps between the bar resolution $B(A,A)$
of $\uk$ and the Green--Happel--Zacharia resolution 
$\Tor_A\otimes_\tau A$ of $\uk$ that are part of a homotopy retract 
datum; see~\cite{GHZ} for details.  Naturally, we have a dual result for the
Yoneda algebra $\Ext_A(\uk,\uk)$ of $A$, which we also record.

\begin{theorem}\label{thm:QMRmin}
For each quiver monomial algebra $A$ there is a minimal model
$B\longrightarrow A$ where $B = \Omega_\infty\!
\Tor_A$, and for a chain $\gamma\in\Tor_A$ the differential $d$ acts by
\[ d\gamma  = -\sum_{n\geqslant 2}
	(-1)^{\binom{n+1}{2}+|\gamma_1|} 
	  \gamma_1 \cdots \gamma_n, 
	 		\]
where the sum ranges through all possible decompositions of $\gamma$.\qed \end{theorem}

\begin{theorem}\label{thm:ExtQ}
There is a canonical $A_\infty$-algebra structure on $\Ext_A$ given as follows.
If $n\in\NN_{\geqslant 2}$ and if $\gamma_1^\vee,\ldots,\gamma_n^\vee$ 
are chains in $\Ext_A$ of respective lengths $(r_1,\ldots,r_n)$, then 
\[ \mu_n(\gamma_1^\vee \otimes\cdots\otimes\gamma_n^\vee)=(-1)^M\gamma^\vee\]
 
if the concatenation $\gamma = \gamma_1\cdots\gamma_n$ is a chain of length 
$r=r_1+\cdots+r_n+1$ where $M$ is the integer $\binom{n+1}2-1+ \sum_{i<j} r_i(r_j+1) + r_1+r$. Otherwise,
this higher product is zero. \qed
\end{theorem}

{{}} Let us remark that the theorem above is a common 
generalisation of the results in \cite{GreenZ} and in~\cite{HeLu},
the latter in the case of monomial algebras. In the first 
the authors describe a multiplicative basis of $\Ext_A$ for $A$ a 
monomial quiver algebra given in terms of Anick chains, and show if
$\gamma_1$ and $\gamma_2$ are chains, then $\gamma_1\smile \gamma_2$
is zero unless the concatenation $\gamma_1\gamma_2$ is a chain, 
in which case $\gamma_1\smile \gamma_2=\gamma_1\gamma_2$. In 
the second, the authors describe the higher products in $\Ext_A$
for monomial algebras that are $p$-Koszul, and show that the chains 
involved in a product $\mu_p(\gamma_1\otimes\cdots\otimes \gamma_p)$ are all 
of odd homological degree. A calculation shows that
the only term that contributes to a sign in the integer $M$ of 
Theorem~\ref{thm:Ext} is the binomial coefficient $\binom{p+1}2$. 
Switching to the sign convention for the Stasheff identities used
in \cite{HeLu} removes this sign, and then our result coincides with 
their result exactly: the higher product of $\gamma_1\otimes\cdots \otimes \gamma_p 
\in \Ext_A^{\otimes n}$ is zero unless the chains $\gamma_1,\ldots,
\gamma_p$ concatenate, in this order, to a chain $\gamma$ of the correct homological degree, in which 
case $\mu_p(\gamma_1\otimes\cdots\otimes \gamma_p) = \gamma$.

\section{Some applications}\label{sec4:Apps}

\subsection{Computation of invariants and operations}
We now use our description of the minimal model of a monomial algebra to 
obtain a model of its Hochschild cochain complex; we refer the reader to \cite{MarklDef}*{2.1} for the definition of this cohomology 
theory and a panorama of its relation to deformation theory, higher structures, and homotopy theory of algebras. If $f : B \longrightarrow B'$
is a map of dga algebras, a map $\partial : B
\longrightarrow B'$ is an \new{$f$-derivation} if
$\partial \mu = \mu(f\otimes \partial + \partial \otimes f)$, and
we write $\Der_f(B,B')$ for the
space of such $f$-derivations. When $B=B'$ and $f$ is
the identity of $B$, we write $\Der(B)$ for such 
space. 
For convenience, we will denote
$\Omega_\infty\!\Tor_A$ by $B$ in what follows. We write
$\HH^*(A)$ for the Hochschild cohomology of an algebra $A$ with
coefficients in itself.
 
{{}} Having obtained a minimal model $\alpha : B\longrightarrow A$ for $A$, we can
produce a cochain complex to compute the Hochschild cohomology of $A$ as
follows. There is a map $\tau : \Tor_A \longrightarrow A$ of degree $-1$ which extends uniquely to the map of algebras $\alpha$, such that 
$\tau[x] = x$ for each variable of $x\in A$.
This is a \emph{twisting cochain} in the sense of~\cite{Pro}: it satisfies the
\emph{Maurer--Cartan equation}
\[\label{eq:MC} \partial \tau + \sum_{n\geqslant 1}(-1)^{\binom n2} \tau^{[n]} = 0 \]
where $\tau^{[n]}: C\to A$ is defined by the composition $\mu^{(n)}\tau^{\otimes n}\Delta_n$. Indeed, $\partial \tau$ is zero since $A$ has trivial differential, and 
for an Anick chain $\gamma$, $\tau^{[n]}(\gamma)$ is zero for trivial reasons
unless $\gamma$ is a $1$-chain of length $n$, in which case $\tau^{[n]}(\gamma)$ is
simply the image of $\gamma$ in $A$, a relation, and is thus zero. Note the Maurer--Cartan equation is equivalent to the fact $\alpha b$ vanishes,
where $b$ is the map of Theorem~\ref{thm:MRmin}.

{{}} From this we obtain the \emph{twisted hom-complex associated to $\tau$}, 
which we denote by $\hom_\tau(\Tor_A,A)$. Its underlying graded vector space 
is $\hom(\Tor_A,A)$, the space of graded $\kk$-linear maps $\Tor_A
\longrightarrow A$, and its differential is obtained as follows. Let 
us write $\mathcal D_A$ for the space of $\alpha$-derivations 
$\Der_\alpha(B,A)$ and $\mathcal T_A$ for the twisted 
chain complex $\hom_\tau(\Tor_A,A)$. Observe that if $f : \Tor_A^0
\longrightarrow A$ is an element of $\mathcal T_A^0$, which amounts to
an element $a \in A$, we have a map $d_f : \Tor_A^1\longrightarrow A$
given by $d_f[x] = [a,x]$, which extends uniquely to a derivation in 
$\mathcal D_A$, and gives us a map $j_A : A \longrightarrow \mathcal D_A$. 
Moreover, if $F \in \mathcal D_A$ is a derivation, the fact that $\alpha d=0$
means that $d^*(F) = (-1)^{|F|-1} Fb$ is an $\alpha$-derivation, and 
$\mathcal D_A$ is then a cochain complex with differential $d^*$. We form the
cone of $j_A$ which we denote by $A\oplus \mathcal D_A[-1]$ and now 
record the following proposition and refer the reader to \cite{BriBen}*{\S 2.3}
for details. 

\begin{prop}
There is an isomorphism $A\oplus \mathcal D_A[-1] \longrightarrow 
\mathcal T_A$  of graded vector spaces that sends a derivation in the domain to 
the suspension of its restriction to $\Tor_A$ and identifies $A$ with
$\hom(\Tor^0_A,A)$. The differential of $\mathcal T_A$ is induced from this 
isomorphism, so that if $f : \Tor_A\longrightarrow A$ is a linear map of nonzero
degree, $df$ is the suspension of the restriction of $d^*(F)$ to $\Tor_A$,
where $F$ is the unique derivation in $\mathcal D_A$ extending $f$. If
$f : \Tor_A^0\longrightarrow A$ is linear, then $df : \Tor_A^1 \longrightarrow A$
is the map given by $x\longmapsto [f[],x]$.
\qed
\end{prop}

{{}}\label{para:HH} The usual Hochschild complex is the twisted complex 
$\hom_\pi(BA,A)$ where $\pi : BA\longrightarrow A$ is the projection onto $A$ from
the bar construction of $A$, with twisted differential 
$  \partial_{BA}^* + [\pi,-] $. The map
$A\oplus \Der(B,A) \longrightarrow	
	A\oplus \Der(\Omega BA,A)$
induced by the homotopy equivalence $B\longrightarrow
\Omega BA$ from Theorem~\ref{thm:HTT},
induces, in turn, a morphism $\hom_\tau(\Tor_A,A)\longrightarrow \hom_\pi(BA,A)$. Since $\Omega_\infty(q)$
is a homotopy equivalence, this map is a quasi-isomorphism, so the cohomology of $\mathcal T_A$ is precisely $\HH^*(A)$.

{{}} The next proposition addresses the computation of cup
products in $\HH^*(A)$ using the complex $\mathcal T_A$
which computes it. We note that, in fact, 
this complex is an $A_\infty$-algebra, and that its multiplication
induces the cup product in Hochschild cohomology. We
refer the reader to \cite{Has}*{Chapter 8, \S 1} for details.

\begin{prop}\label{prop:highercups}
For each $n\in\NN_{\geqslant 2}$, define a higher product 
$ \mu_n : \mathcal T_A^{\otimes n} \longrightarrow \mathcal T_A$ so that for
linear maps $f_1,\ldots,f_n\in \mathcal T_A$, $ \mu_n(f_1\otimes\cdots\otimes f_n)
(\gamma) =(-1)^N \mu^{(n)}_A(f_1\otimes\cdots\otimes f_n)\Delta_n(\gamma)$, where
we set $N=n(|f_1|+\cdots+|f_n|+1)$. These maps define on $\mathcal T_A$ an 
$A_\infty$-algebra structure, and on cohomology the map
$\mu_2$ induces the cup product of $\HH^*(A)$. \qed
\end{prop}

{{}} It is fair to observe that the construction of our minimal
model requires the 
construction of a homotopy retract datum from $BA$ to $\Tor_A$,
and thus of comparison morphisms, which are usually difficult to produce. However, the construction of this
retraction is streamlined by the machinery of algebraic discrete 
Morse theory and, in fact, one may attempt to apply the methods 
outlined in \cite{ADMT} to any algebra admitting a Gr\"obner basis
to produce a model of it. Let us also remark that one need
not recourse to comparison maps to produce models of algebras. In the
article \cite{DotKho}, for example, the authors produce models for
monomial operads, in particular for monomial algebras, without
doing this. As explained in that article, one may use this model to 
understand not necessarily monomial algebras admitting a
Gr\"obner basis by the method of homological perturbation theory.
Remark, too, that in~\cite{ReRo} the authors produce chain 
comparison maps between the Bardzell resolution of a monomial quiver 
algebra and its usual bar resolution, and succeed in using them to compute 
the Gerstenhaber bracket on Hochschild cohomology of some examples. It may 
be the case that the maps of \cite{ReRo} are a part of a homotopy 
retract datum provided by algebraic discrete Morse theory \cite{ADMT,Skol}. 

\subsection{Computation of Tamarkin--Tsygan calculi}

We noted that the twisted complex $\mathcal T_A$ is naturally
isomorphic to the complex $A\oplus \mathcal D_A[-1]$. The morphism $\alpha :
B\longrightarrow A$ induces a map $\alpha\oplus\alpha_* : 
B\oplus \Der(B)[-1]\longrightarrow A\oplus \mathcal D_A[-1]$ 
by post-composition, which one can check is a quasi-isomorphism. The domain of 
this map is, naturally, a dg Lie algebra, whose cohomology is $\HH^*(A)$, and it 
is not hard to prove its Lie bracket induces the Gerstenhaber bracket on 
$\HH^*(A)$, which gives a description of the Gerstenhaber bracket of $A$ in terms of a model, without having recourse to the bar construction of $A$ or comparison morphisms. It seems the first intrinsic definition of the Gerstenhaber
bracket was given in~\cite{Sheff} by Stasheff, where it is shown, among
other things, that
the Lie bracket in the complex $\Coder(BA)$ of coderivations of the bar 
construction of $A$ induces the Gerstenhaber bracket on Hochschild cohomology.

It is important to note that the computation of 
$\HH^*(A)$ through this dg Lie algebra is plausible, for example,
if the model has finitely many generators; see \cite{FMT} for
two examples. In the case of monomial quiver algebras, it may very well
happen that, although $\Tor_A$ is locally finitely dimensional, it is
not finitely dimensional. There is, however, hope that computing
Hochschild cohomology, and thus the Gerstenhaber bracket, using derivations of a 
minimal model is feasible. Let us mention, too, that one can also compute
cyclic
homology and non-commutative de Rham homology of $A$ through a model 
following~\cite{FeiginTsygan}, using non-commutative differential forms.
These are treated in detail, for example, in~\cite{Karoubi}*{Chapter 1} 
and~\cite{Loday}*{Chapter 2,\S 6}. 

One can
in fact compute the Tamarkin--Tsygan calculus~\cite{TT} of $A$ through a model; \textcolor{comcol}{we have
pursued this in~\cite{TTC}, where in particular we 
use this minimal model to compute the 
Tamarkin--Tsygan calculus of some monomial
algebras.}
\medskip

\subsection{Support variety theory for Gorenstein monomial algebras.}
In joint work with V. Dotsenko and V. Gelinas~\cite{Gorenstein}, we used the higher structure
on $\Tor_A$ obtained here and the
notion of higher centres of Briggs--Gelinas to 
deduce that a monomial algebra satisfies
the FG conditions of Snashall--Soldberg if and
only if it is Gorenstein. We also showed that
in this case, if the algebra is of Gorenstein
dimension $d$, there is a periodicity operator
on Hochschild cohomology whose cup product map
induces isomorphism in degrees above $d$, and that
its Tate--Hochschild cohomology is given by its periodic
Hochschild cohomology: it is simply obtained by
inverting this operator in Hochschild cohomology.

\subsection{The case of algebras with a Gröbner basis}

{{}} Let us put ourselves in the situation
where $A$ is a finitely generated algebra with
generators $V$ and ideal of relations $(R)$.
Pick a Gröbner basis with respect to a monomial 
order on $TV$, and let us write $A'$ for the monomial
algebra associated to $A$ and $B' = (TW,d')$ for
the minimal model of Theorem~\ref{thm:MRmin}. 
Note that since $W$ consists of monomials of $TV$,
this graded space is partially ordered by looking
at the support of a chain, and this order extends to
monomials lexicographically.

{{}} 
We claim that there exists a model $B = (TW,d)$ 
of $A$ such that for any $w\in W$, the terms 
appearing in $(d-d')(w)$ are smaller than $w$,
and such that the associated graded morphism
to $B\longrightarrow A$ is the model 
$B'\longrightarrow A'$ in the main theorem of
these notes. As before, let $(C,d)$ denote the
complex obtained from the Anick resolution
of $A$ that computes $\Tor_A$. Note that
Proposition~\ref{prop:flushright} is still
valid if we replace $\Tor_A$ with $C$, since 
at no point we used $A$ is monomial
to prove it. We also observe that the 
differential on $B'$ preserves the support
of a chain.

{{}} Naturally, to prove our claim,
it suffices we do it for each higher 
coproduct, including the possibly non-zero
differential $\Delta_1'$ on $C$. The work 
of Anick shows this differential decreases
the order of a chain, and the claim is obvious 
for $\Delta_2'$,
so we may only worry about $\Delta_n'$ for 
$n\in\NN_{\geqslant 3}$. In this case,
the recursive formula of Proposition~\ref{prop:flushright}
means it suffices we do this for the homotopy
$h$. But this follows from the fact it is built
from the differential of $BA$, which, after
rewriting possible non-zero products that appear,
decreases the order of the underlying monomial
of any bar term, independent of them being 
a cycle or not. From this we obtain the
desired result:

\begin{theorem}
Let $A$ be a finitely generated algebra with
a finite Gröbner basis, and
let $A'$ be its associated monomial algebra.
There exists a (possibly non-minimal) model $(B,d)\longrightarrow A$ such
that the associated graded morphism $(B,d')\longrightarrow A'$ is the model of Theorem~\ref{thm:MRmin}. More precisely, we can arrange
it so that $d-d'$ decreases the order of the
underlying monomial of a chain in $B$.
\end{theorem}

\begin{proof}
We have given some
details in the discussion preceding the
statement of the theorem to obtain a proof
following the strategy used to prove our
main theorem. Alternative, one can use a
homological perturbation argument completely analogous to~\cite{DotKho}*{Theorem 4.1},
where instead of starting with the (usually
non-minimal) model of the authors, one starts
with the minimal model of our main theorem
with the internal grading given by the
underlying monomial of an Anick chain.
\end{proof}
\noindent We remark that this theorem is not too 
surprising, since it is the non-linear analog of 
the work of S. Chohuy in his PhD thesis~\cite{Sergio},
with A. Solotar. The lack of an explicit formula for 
the perturbed differential makes this theorem 
uninteresting for computations: in concrete examples,
what we 
usually do is produce a perturbed differential
which squares to zero, since it is usually 
possible to come up with a candidate
of model and, through a filtration argument,
show it is indeed acyclic. However, we would like
to state the following

\begin{conjecture*}
Let $A$ be as before, and let $w$ be a chain
in the generators of the model $(B,d)\longrightarrow A$. Then the basis elements 
appearing in $dw$ are obtained as follows:
\begin{enumerate}[label = \emph{\textcolor{newcol}{\textbf{C\arabic*.}}}]
\setlength{\itemsep}{0pt}
  \setlength{\parskip}{0pt}\item Compute all possible decompositions of the
chain $w$.
\item Starting from the left, rewrite the chain 
$w$ once, and obtain all possible decompositions
into chains of the terms that appear after this.
\item Repeat this procedure until all terms
that appear are in normal form.
\end{enumerate}
\end{conjecture*}

{{}} As an example, let us consider the
algebra $J$ with two generators $x$ and $y$ 
subject to the relations $x^2=0$ and 
$y^2x=xy^2+xyx$, and lexicographical order with respect to $y>x$. The associated monomial algebra
$J'$ has relations $x^2=0$ and $y^2x=0$, and its
model has generators $x_0,y_0,x_1,y_1\ldots$ with
differential
\[d y_{n+1} =y^2x_n	+\sum_{\substack {s+t=n \\s\geqslant 1}}
					(-1)^sy_sx_t, \quad
					d x_{n+1} = \sum_{s+t=n} (-1)^sx_sx_t. \]
Here, for $n\in\NN$, the generator $y_n$ has
underlying ambiguity $y^2x^n$ while $x_n$
has underlying ambiguity $x^{n+1}$, which
our differential preserves. 
The differential then codifies all possible  
2-decompositions of $y^2x^n$ into $y^2x^s\cdot
x^t$ for $s+t=n$, and the unique $3$-decomposition $y\cdot y\cdot x^n$. Similarly,
$x^{n+2}$ only admits $2$-decompositions of
the form $x^{s+1}\cdot x^{t+1}$ where $s+t=n$. 
The model 
corresponding to the original algebra $J$ incorporates lower order
terms as follows:
\begin{align*}
d y_{n+1} = [y^2,x_n] -\sum_{s+t=n} x_s y x_t
	-\sum_{\substack {s+t=n \\t\geqslant 1}}
					(x_sy_t-(-1)^ty_tx_s),
					\quad
d x_{n+1} = \sum_{s+t=n} (-1)^sx_sx_t.
\end{align*}

{{}} It is routine to check this perturbed differential
squares to zero, so that we have obtained a model
of $J$. To illustrate our conjecture, let us
consider the term $y_2 = y^2x^2$. This can be
decomposed into the chains $y_0^2x_1$ and $y_1x_0$
and no others. Rewriting, we obtain two terms,
$xy^2x$ and $xyx^2$. The first can be decomposed
into $x_0y_1$ only, and the second into $xyx_1$.
We can only rewrite the first monomial, and 
we obtain $x^2y^2$ and $x^2yx$ which rewrite to zero. We can decompose
these into $x_1y^2$ and $x_1y_0x_0$, and no other
terms. Summing up, the basis elements that appear 
are the following:
$ y^2x_1$, $x_1y^2$,  $y_1x_0$,  $x_0y_1$,  $x_1yx_0$,  $x_1yx_0$.  
These are precisely those appearing in the formula
 for $dy_2$ above.
 \begin{multicols}{2}

\Addresses
\end{multicols}


\begin{bibdiv}
\begin{biblist}
\bib{Anick}{article}{
   author={Anick, D. J.},
   title={On the homology of associative algebras},
   journal={Trans. Amer. Math. Soc.},
   volume={296},
   date={1986},
   number={2},
   pages={641--659},
   issn={0002-9947},
   review={\MR{846601}},
}
\bib{AnickG}{article}{
   author={Anick, D. J.},
   author={Green, E. L.},
   title={On the homology of quotients of path algebras},
   journal={Comm. Algebra},
   volume={15},
   date={1987},
   number={1-2},
   pages={309--341},
   issn={0092-7872},
   review={\MR{876982}},
}
\bib{BriBen}{article}{
   author={Briggs, B.},
   author={Gelinas, V.},
   title={The $A_\infty$-centre of the Yoneda Algebra and the Characteristic Action of Hochschild Cohomology on the Derived Category},
   date={2017},
   eprint={arXiv:1702.00721}
}
\bib{Sergio}{article}{
   author={Chouhy, S.},
   author={Solotar, A.},
   title={Projective resolutions of associative algebras and ambiguities},
   journal={J. Algebra},
   volume={432},
   date={2015},
   pages={22--61},
   issn={0021-8693},
   review={\MR{3334140}},
   doi={10.1016/j.jalgebra.2015.02.019},
}
\bib{Cibils}{article}{
   author={Cibils, C.},
   title={Rigidity of truncated quiver algebras},
   journal={Adv. Math.},
   volume={79},
   date={1990},
   number={1},
   pages={18--42},
   issn={0001-8708},
   review={\MR{1031825}},
   doi={10.1016/0001-8708(90)90057-T},
}
\bib{Gorenstein}{article}{
   author={Dotsenko, V.},
   author={Gelinas, V.},
   author={Tamaroff, P.},
   title={Finite generation for Hochschild cohomology of Gorenstein monomial algebras},
   eprint = {arXiv:1909.00487 [math.KT]},
   date={2019},
   pages={38},
}
\bib{DotKho}{article}{
   author={Dotsenko, V.},
   author={Khoroshkin, A.},
   title={Quillen homology for operads via Gr\"obner bases},
   journal={Doc. Math.},
   volume={18},
   date={2013},
   pages={707--747},
   issn={1431-0635},
   review={\MR{3084563}},
}
\bib{MarklDef}{article}{
   author={Doubek, M.},
   author={Markl, M.},
   author={Zima, P.},
   title={Deformation Theory (Lecture Notes)},
   journal={Archivum mathematicum 43(5)},
   eprint = {arXiv:0705.3719 [math.AG]},
   date={2007},
   pages={333-371},
}
\bib{FMT}{article}{
   author={F\'elix, Y.},
   author={Menichi, L.},
   author={Thomas, J.-C.},
   title={Gerstenhaber duality in Hochschild cohomology},
   journal={J. Pure Appl. Algebra},
   volume={199},
   date={2005},
   number={1-3},
   pages={43--59},
   issn={0022-4049},
   review={\MR{2134291}},
   doi={10.1016/j.jpaa.2004.11.004},
}
\bib{FeiginTsygan}{book}{
   author={Feigin, B.},
   author={Tsygan, B.},
   subtitle={Cyclic homology of algebras with quadratic relations, 
universal enveloping algebras and group algebras},
   title={K-Theory, Arithmetic and Geometry},
   series={Lecture Notes in Mathematics},
   publisher={Springer--Verlag},
   date={1987},
   volume={1289},
   pages={210--239},
   }
\bib{GHZ}{article}{
   author={Green, E.},
   author={Happel, D.},
   author={Zacharia, D.},
   title={Projective resolutions over Artin algebras with zero relations},
   journal={Illinois J. Math.},
   volume={29},
   date={1985},
   number={1},
   pages={180--190},
   issn={0019-2082},
   review={\MR{769766}},
}
\bib{GreenZ}{article}{
   author={Green, E.},
   author={Zacharia, D.},
   title={The cohomology ring of a monomial algebra},
   journal={Manuscripta Math.},
   volume={85},
   date={1994},
   number={1},
   pages={11--23},
   issn={0025-2611},
   review={\MR{1299044}},
}
\bib{HeLu}{article}{
   author={He, J.-W.},%*{inverted={yes}},
   author={Lu, D.-M.},%*{inverted={yes}},
   title={Higher Koszul algebras and $A_\infty$-algebras},
   journal={J. Algebra},
   volume={293},
   date={2005},
   number={2},
   pages={335--362},
   issn={0021-8693},
   review={\MR{2172343}},
}
\bib{Hinich}{article}{
   author={Hinich, V.},
   title={Homological algebra of homotopy algebras},
   journal={Comm. Algebra},
   volume={25},
   date={1997},
   number={10},
   pages={3291--3323},
   issn={0092-7872},
   review={\MR{1465117}},
   doi={10.1080/00927879708826055},
}
\bib{ADMT}{article}{
   author={J\"ollenbeck, M.},
   author={Welker, V.},
   title={Minimal resolutions via 
algebraic discrete Morse theory},
   journal={Mem. Amer. Math. Soc.},
   volume={197},
   date={2009},
   number={923},
   pages={vi+74},
   issn={0065-9266},
   isbn={978-0-8218-4257-7},
   review={\MR{2488864}},
}
\bib{Karoubi}{article}{
   author={Karoubi, M.},
   title={Homologie cyclique et $K$-th\'eorie},
   language={French, with English summary},
   journal={Ast\'erisque},
   number={149},
   date={1987},
   pages={147},
   issn={0303-1179},
   review={\MR{913964}},
}
\bib{Kell}{article}{
   author={Keller, B.},
   title={$A_\infty$-algebras, modules and functor categories},
   conference={
      title={Trends in representation theory of algebras and related topics},
   },
   book={
      series={Contemp. Math.},
      volume={406},
      publisher={Amer. Math. Soc., Providence, RI},
   },
   date={2006},
   pages={67--93},
   review={\MR{2258042}},
}
\bib{Has}{article}{
   author={Lef\`evre-Hasegawa, K.},
   title={ Sur les A-infini catégories, \emph{PhD Thesis}},
   date={2003},
   eprint={arXiv:math/0310337 [math.CT]},
   pages={230},
}
\bib{Loday}{book}{
   author={Loday, J.-L.},
   title={Cyclic homology},
   series={Grundlehren der Mathematischen Wissenschaften [Fundamental
   Principles of Mathematical Sciences]},
   volume={301},
   note={Appendix E by Mar\'\i a O. Ronco},
   publisher={Springer-Verlag, Berlin},
   date={1992},
   pages={xviii+454},
   isbn={3-540-53339-7},
   review={\MR{1217970}},
   doi={10.1007/978-3-662-21739-9},
}
\bib{LV}{book}{
   author={Loday, J.-L.},
   author={Vallette, B.},
   title={Algebraic operads},
   series={Grundlehren der Mathematischen Wissenschaften [Fundamental
   Principles of Mathematical Sciences]},
   volume={346},
   publisher={Springer, Heidelberg},
   date={2012},
   pages={xxiv+634},
   isbn={978-3-642-30361-6},
   review={\MR{2954392}},
   doi={10.1007/978-3-642-30362-3},
}
\bib{Markl}{article}{
   author={Markl, M.},
   title={Transferring $A_\infty$ 
(strongly homotopy associative) 
structures},
   journal = {Rend. Circ. Mat. Palermo (2) Suppl.},
     number = {79},
      year = {2006},
     pages = {139--151},
      issn = {1592-9531},
}
\bib{Pro}{article}{
   author={Prout\'e, A.},
   title={$A_\infty$-structures, \emph{PhD Thesis}},
   journal = {Reprints in Theory and Applications of Categories, No. 21},
   date={2011},
   pages={1--99}
}
\bib{ReRo}{article}{
   author={Redondo, M. J.},
   author={Rom\'an, L.},
   title={Comparison Morphisms Between Two Projective Resolutions of Monomial Algebras},
   journal = {Revista de la Uni\'on Matem\'atica Argentina},
   date={2018},
   pages={1--31},
}
\bib{Singer}{article}{
   author={Singer, W. M.},
   title={On the minimality of Anick's resolution, with application to an
   algebra of cohomology operations},
   journal={J. Algebra},
   volume={179},
   date={1996},
   number={3},
   pages={918--929},
   issn={0021-8693},
   review={\MR{1371751}},
   doi={10.1006/jabr.1996.0044},
}
\bib{Skol}{article}{
   author={Sk\"oldberg, E.},
   title={Morse theory from an algebraic viewpoint},
   journal={Trans. Amer. Math. Soc.},
   volume={358},
   date={2006},
   number={1},
   pages={115--129},
   issn={0002-9947},
   review={\MR{2171225}},
   doi={10.1090/S0002-9947-05-04079-1},
}
\bib{Sheff}{article}{
   author={Stasheff, J.},
   title={The intrinsic bracket on the deformation complex of an associative
   algebra},
   journal={J. Pure Appl. Algebra},
   volume={89},
   date={1993},
   number={1-2},
   pages={231--235},
   issn={0022-4049},
   review={\MR{1239562}},
   doi={10.1016/0022-4049(93)90096-C},
}
\bib{TT}{article}{
   author={Tamarkin, D.},
   author={Tsygan, B.},
   title={The ring of differential operators on forms in noncommutative
   calculus},
   date={2005},
   pages={105--131},
   review={\MR{2131013}},
   doi={10.1090/pspum/073/2131013},
}
\bib{TTC}{article}{
   author={Tamaroff, P.},
   title={The Tamarkin--Tsygan calculus of an algebra a la Stasheff},
   journal = {Homol. Homotopy Appl.},
   date={2021},
   pages={26},
   volume={23(1)},
}
\bib{HTHA}{article}{
   author={Vallette, B.},
   title={Homotopy theory of homotopy algebras},
   date={2014},
   pages={32},
   eprint = {arXiv:1411.5533 [math.AT]},
}
\end{biblist}
\end{bibdiv}
\end{document}